\documentclass[11pt,oneside,english]{siamltex}
\usepackage[T1]{fontenc}
\usepackage[latin9]{inputenc}
\usepackage{geometry}
\usepackage{xcolor}
\usepackage{pdfcolmk}
\usepackage{mathrsfs}
\usepackage{amsbsy}
\usepackage{amstext}
\usepackage{amssymb}
\usepackage{graphicx}
\usepackage{setspace}
\PassOptionsToPackage{normalem}{ulem}
\usepackage{ulem}
\usepackage{nomencl}
\providecommand{\printnomenclature}{\printglossary}
\providecommand{\makenomenclature}{\makeglossary}
\makenomenclature
\makeatletter

\providecolor{lyxadded}{rgb}{0,0,1}
\providecolor{lyxdeleted}{rgb}{1,0,0}
\newcommand\eqref[1]{(\ref{#1})}

\usepackage{savesym}
\savesymbol{eqref}
\usepackage[version=3]{mhchem}
\restoresymbol{dmb}{eqref}
\usepackage{cite}
\usepackage{mathrsfs} 
\usepackage{pgf}
\usepackage{tikz}
\usetikzlibrary{arrows,automata}
\usepackage{arydshln, bbm,bm,booktabs,floatrow,geometry,subfig,blkarray}

\newtheorem{rem}[theorem]{Remark}

\@ifundefined{showcaptionsetup}{}{%
 \PassOptionsToPackage{caption=false}{subfig}}
\usepackage{subfig}
\makeatother

\usepackage{babel}

\begin{document}

\title{Analytical Equilibrium Solutions of Biochemical Systems with Synthesis
and Degradation }

\author{I.~Mirzaev\thanks{Applied Mathematics, Univeristy of Colorado, Boulder, CO 80309-0526}
\and D.~M.~Bortz$^{*}$\thanks{Corresponding author (dmbortz@colorado.edu)}}
\maketitle
\begin{abstract}
Analyzing qualitative behaviors of biochemical reactions using its
associated network structure has proven useful in diverse branches
of biology. As an extension of our previous work, we introduce a graph-based
framework to calculate steady state solutions of biochemical reaction
networks with synthesis and degradation. Our approach is based on
a labeled directed graph \nomenclature{$G$}{Labeled directed graph (or simply digraph)}$G$
and the associated system of linear non-homogeneous differential equations
with first order degradation and zeroth order synthesis. We also present
a theorem which provides necessary and sufficient conditions for the
dynamics to engender a unique stable steady state. 

Although the dynamics are linear, one can apply this framework to
nonlinear systems by encoding nonlinearity into the edge labels. We
answer open question from our previous work concerning the non-positiveness
of the elements in the inverse of a perturbed Laplacian matrix. Moreover,
we provide a graph theoretical framework for the computation of the
inverse of a such matrix. This also completes our previous framework
and makes it purely graph theoretical. Lately, we demonstrate the
utility of this framework by applying it to a mathematical model of
insulin secretion through ion channels and glucose metabolism in pancreatic
$\beta$-cells.
\end{abstract}

\section{Introduction}

In recent years, many researchers have devoted their efforts to developing
a systems-level understanding of biochemical reaction networks. In
particular, the study of these chemical reaction networks (CRNs) using
their associated graph structure has attracted considerable attention.
The work led by Craciun and Feinberg on multistationarity \cite{Craciun2005,Craciun2010,Craciun2006,CraciunFeinberg2006}
and the work led by Mincheva and Roussel on stable oscillations \cite{Mincheva2011a,Mincheva2007a,Mincheva2007}
are two particularly influential approaches. For a good overview of
the various graph theoretic developments, we direct the interested
reader to the review provided in Domijan and Kirkilionis \cite{Domijan2008}.

In this work, we focus on applications of graph theory, mainly the
Matrix-Tree Theorem (MTT), for deriving equilibrium \nomenclature{ES}{Equilibrium solution (Steady state solution)}solutions
(ES) for CRNs that fit within a Laplacian dynamics framework. The
MTT-based framework was first applied in a biological context by King
and Altman \cite{King1956} to derive steady state rate equations
in enzyme kinetics. This framework was then simplified and summarized
into rules (known as \emph{Chou's graphical rules} \cite{Lin2013})
by Chou and coworkers \cite{Chou1989,Chou1990,Chou1993}. Chou \cite{Chou1989}
has also extended the framework for non-steady state enzyme-catalyzed
systems. 

The main disadvantage of Chou's graphical rules is that they are only
applicable if the underlying digraph structure is \emph{strongly connected},
i.e., every vertex is reachable from every other vertex. This issue
was solved and extended for general directed graphs (\emph{digraphs})
by Mirzaev and Gunawardena in 2013 \cite{MirzaevGunawardena2013bmb}
and is applicable to the specific class of linear ordinary differential
equations (ODEs) known as \emph{Laplacian dynamics}. Systems described
by Laplacian dynamics are created using a weakly connected digraph,
$G$, with $n$ vertices, with labelled, directed edges, and without
self loops. Note that by\emph{ weakly connected }we mean that the
graph cannot be expressed as the union of two disjoint digraphs. If
there is an edge from vertex $j$ to vertex $i$, we label it with
\nomenclature{$e_{ij}$}{an edge from vertex $j$ to vertex $i$}$e_{ij}>0$,
and with $e_{ij}=0$ if there is no such edge. %
\footnote{If a negative edge weight is encountered in applications, one can
reverse orientation of that edge, hence preserving positivity of edge
labels.%
}

The Laplacian matrix (hereafter, a \emph{Laplacian} $\mathcal{L}$)
\nomenclature{$\mathcal{L}(G)$}{Laplacian matrix of digraph $G$}of
given digraph $G$ is then defined as 

\begin{equation}
\left(\mathcal{L}(G)\right)_{ij}=\begin{cases}
e_{ij} & \text{if }i\ne j\\
-\sum_{m\ne j}e_{mj} & \text{if }i=j\,.
\end{cases}\label{eq:Laplacian Matrix}
\end{equation}
The corresponding \emph{Laplacian} \emph{dynamics} are then defined
as \nomenclature{$\frac{d\mathbf{x}}{dt}=\mathcal{L}(G)\mathbf{x}$}{Laplacian dynamics defined on digraph $G$}
\[
\frac{d\mathbf{x}}{dt}=\mathcal{L}(G)\cdot\mathbf{x}
\]
where $\mathbf{x}=\left(x_{1},\cdots,x_{n}\right)^{T}$ is column
vector of species concentrations at each vertex, $1,\cdots,n$. In
a biochemical context one may think of vertices as different species
and edges as rate of transformation from one species to another. However,
we note that this framework is symbolic in nature in the sense that
the mathematical description of the computed steady states is done
without the specification of rate constants, i.e., edge weights $e_{ij}$.
In other words, the only information about an individual $e_{ij}$
relevant to our approach is whether or not it is zero.

Laplacian matrices were first introduced by Kirchhoff in 1847 in his
article about electrical networks \cite{Kirchhoff1847}. Ever since
then Laplacians have been studied and applied in various fields. For
an example of studying the applications of Laplacians to spectral
theory, we refer the interested reader to Bronski and Deville \cite{Bronski2014}
in which they study the class of \emph{Signed graph Laplacians} (a
symmetric matrix, which is special case of above defined Laplacian).

In this article we will extend the framework intitially developed
in \cite{MirzaevGunawardena2013bmb} to investigate behaviors of Laplacian
dynamics when zero-th order synthesis and first order degradation
are added to the system. Specifically, we will examine the following
dynamics,\nomenclature{$\frac{d\mathbf{x}}{dt}=\mathcal{L}(G)\mathbf{x}-D\mathbf{x}+\mathbf{s}$}{Synthesis and degradation dynamics}

\begin{equation}
\frac{d\mathbf{x}}{dt}=\mathcal{L}(G)\cdot\mathbf{x}-D\cdot\mathbf{x}+\mathbf{s}\label{eq:New dynamics}
\end{equation}
where the degradation matrix $D$ is a diagonal matrix with \nomenclature{$d_i$}{Label of degration edge at vertex $i$}$\left(D\right)_{ii}=d_{i}\geq0$
and the synthesis vector $\mathbf{s}$ is a column vector with \nomenclature{$s_i$}{Label of synthesis edge at vertex $i$}$\left(\mathbf{s}\right)_{i}=s_{i}\geq0$.
\nomenclature{$D$}{Degradation matrix, which is a diagonal matrix with degradation edges as diagonal entries}\nomenclature{$\mathbf{s}$}{Synthesis vector: a column vector with synthesis edges as entries}Hereafter,
we refer to this new dynamics as synthesis and degradation dynamics
(or simply as SD dynamics). In the biological networks literature
this type of dynamics are often referred as \emph{inconsistent }networks\cite{Marashi2014}. 

For these dynamics, several questions naturally arise. Under what
conditions does this system have non-negative, stable ES solution?
Moreover, how can we relate the ES solution to the underlying digraph
structure of $G$ as we did for Laplacian dynamics without synthesis
and degradation? Our goal is to answer these questions on a theoretical
level as well as apply the result to real world CRN examples.

The outline of this work is as follows. We will first briefly review
the main results of \cite{Gunawardena2012,MirzaevGunawardena2013bmb}
and present some additional notation (to be used in subsequent sections).
In Section \ref{sec:Theoretical-Development} we describe our main
theoretical results and in Section \ref{sec: negativity of inverse}
fully discuss the proof of an important result in Section \ref{sec:Theoretical-Development}.

In Section \ref{sec:Biochemical-Network-Application}, we illustrate
an application of these results to exocytosis cascade of insulin granules
and glucose metabolism in pancreatic $\beta$-cells. Lastly, in Section
\ref{sec:Conclusions}, we conclude with a discussion of the implications
of these results as well as plans for future work.%
\footnote{ For the convenience of the reader and to promote clarity, we include
at the end of this document a list of nomenclature used throughout
this work.%
}

\section{\label{sec:Preliminary-results}Preliminaries}

In this section we briefly summarize the important results of Mirzaev
and Gunawardena \cite{MirzaevGunawardena2013bmb} and refer the interested
reader to that article for proofs and more extensive discussion and
interpretation. For the sake of clarity, we will preserve the original
notation while we include some additional definitions that can be
found in many introductory graph theory books.

Given a digraph $G$, we denote the set of vertices of $G$ with $\mathcal{V}(G)$
and we write \nomenclature{$i\Longrightarrow j$}{There exists a path from vertex $i$ to vertex $j$}$i\Longrightarrow j$
to denote the existence of a path from vertex $i$ to\emph{ }vertex
$j$. If $i\Longrightarrow j$ and $j\Longrightarrow i$, vertex $i$
is said to be \emph{strongly connected} to vertex $j$, and is denoted
\nomenclature{$i\Longleftrightarrow j$}{There exists a path from vertex $i$ to vertex $j$, and a path from vertex $j$ to vertex $i$}$i\Longleftrightarrow j$.
A digraph $G$ is \emph{strongly connected} if for each ordered pair
$i,j$ of vertices in $G$, we have that $i\Longleftrightarrow j$.
The \emph{strongly connected components} \nomenclature{SCC}{Strongly connected component}(SCCs)
of a digraph are the largest strongly connected subgraphs. Let \nomenclature{$C[i]$}{ SCC containing $i$}$C[i]$
denote the SCC containing $i$, $i\in\mathcal{V}(C[i])$. Suppose
we are given two SCCs, $C[i]$ and $C[j]$, if $i\Longrightarrow j$
we write \nomenclature{$C[i]\preceq C[j]$}{ SCC containing vertex $j$ can be reached from  SCC containing $i$}$C[i]\preceq C[j]$
to denote that $C[i]$ \emph{precedes} $C[j]$. This \emph{precedes
}relation is both reflexive and transitive. Moreover, the relation
is also antisymmetric as $C[i]\preceq C[j]$ and $C[j]\preceq C[i]$
imply that $i\Longleftrightarrow j$ and $C[i]=C[j]$. From this,
we can conclude that the precedes relation allows for a \emph{partial
ordering} of the SCCs. Accordingly, this allows us to identify so-called
\nomenclature{tSCC}{Terminal strongly connected component}terminal
SCCs (tSCC), which are those SCCs $C[i]$ such that, if $C[i]\preceq C[j]$
then $C[i]=C[j]$. These tSCCs are used in many other contexts, for
example, they are also known as ``attractors'' of state transition
graphs \cite{Berenguier2013}.

With this terminology, we can devise an insightful relabeling of the
vertices of digraph $G$. Such a relabeling will transform the Laplacian
matrix into one with a block lower-diagonal structure, which will
prove convenient in our theoretical development. Suppose there are
$q$ tSCCs out of a total of $p+q$ SCCs. Our goal is to relabel the
vertices such that the first $p$ blocks of Laplacian matrix correspond
to the $p$ non-terminal SCCs. Since the precedence relation, $\preceq$,
is a partial ordering, there exists an ordering of the SCCs, $C_{1},\dots,C_{p+q}$,
such that, \nomenclature{$C_{i}\preceq C_{j}$}{SCC $C_{j}$ can be reached from SCC $C_{i}$.}if
$C_{i}\preceq C_{j}$, then $i\le j$. Since a tSCC cannot precede
any other SCC, then the tSCCs can be in some arbitrary order $\{C_{i}\}_{i=p+1}^{p+q}$
(which will not impact our results). We denote \nomenclature{$a_i$}{Number of vertices in SCC $C_i$}$a_{i}$
as the number of vertices in $C_{i}$, and \nomenclature{$m_i$}{$i^{th}$ partial sum of $c_i$'s,  $m_{i}=\sum_{k=1}^{i}c_{k}$}$m_{i}=\sum_{k=1}^{i}a_{k}$
as the partial sum of the $a_{i}$'s, (with $m_{0}=0$). Note that
the $a_{i}$'s should add up to the number of vertices in digraph
$G$, i.e. $\sum_{i=1}^{p+q}a_{i}=n$. Then the vertices of $C_{i}$
are relabeled using only indices $m_{i-1}+1,\dots,m_{i-1}+a_{i}$
for $i=1,\dots,p+q$. Consequently, the new Laplacian matrix, $\mathcal{L}(G),$
is constructed using the relabeled vertices. Since $i<j$ implies
$C_{j}\not\preceq C_{i}$, the Laplacian of $G$ can be written in
block lower-triangular form

\[
\mathcal{L}(G)=\left(\begin{array}{ccc|ccc}
\boxed{\mathcal{L}_{1}} & \mathbf{0} & \mathbf{0} & \mathbf{0} & \cdots & \mathbf{0}\\
\vdots & \ddots & \vdots & \vdots & \vdots & \vdots\\
+ & + & \boxed{\mathcal{L}_{p}} & \mathbf{0} & \cdots & \mathbf{0}\\
\hline + & \cdots & + & \boxed{\mathcal{L}_{p+1}} & \mathbf{0} & \mathbf{0}\\
\vdots & \vdots & \vdots & \vdots & \ddots & \vdots\\
+ & \cdots & + & \mathbf{0} & \mathbf{0} & \boxed{\mathcal{L}_{p+q}}
\end{array}\right)=\left(\begin{array}{c|c}
N & \mathbf{0}\\
\hline B & T
\end{array}\right)\,,
\]
 where $+$ stands for some matrix with non-negative real entries,
the submatrix $N$ is block lower-triangular with non-negative off-diagonal
elements, $B$ is a matrix with non-negative elements, $\mathbf{0}$
is matrix of all zeros, and $T$ is also a block diagonal matrix such
that\nomenclature{$N$}{Lower-block diagonal submatrix of $\mathcal{L}(G)$ corresponding to non-terminal SCCs}\nomenclature{$T$}{Block diagonal submatrix of $\mathcal{L}(G)$ corresponding to tSCCs}
\begin{equation}
N=\left(\begin{array}{ccc}
\boxed{\mathcal{L}_{1}} &  & \mathbf{0}\\
\vdots & \ddots\\
+ & + & \boxed{\mathcal{L}_{p}}
\end{array}\right),\, T=\left(\begin{array}{ccc}
\boxed{\mathcal{L}_{p+1}} &  & \mathbf{0}\\
 & \ddots\\
\mathbf{0} &  & \boxed{\mathcal{L}_{p+q}}
\end{array}\right)\,.\label{eq: Definition of N}
\end{equation}
By the definition of the Laplacian matrix (see (\ref{eq:Laplacian Matrix}))
all off-diagonal elements are non-negative real numbers. The blocks
in boxes on the main diagonal, denoted by $\mathcal{L}_{1},\dots,\mathcal{L}_{p+q}$,
are the submatrices defined by restricting $\mathcal{L}(G)$ to the
vertices of the corresponding SCCs, $C_{1},\dots,C_{p+q}$. Note that
for $i=p+1,\dots,p+q$ each $\mathcal{L}_{i}$ is Laplacian matrix
in its own, $\mathcal{L}_{i}=\mathcal{L}(C_{i})$ . However for the
non-terminal SCCs, $\left\{ C_{i}\right\} _{i=1}^{p}$, there is always
at least one outgoing edge to some other SCC. This implies that for
$i=1,\dots,p$ each matrix $\mathcal{L}_{i}$ is defined as the Laplacian
of a corresponding SCC minus some non-zero diagonal matrix corresponding
to outgoing edges from this SCC, \nomenclature{$\mathcal{L}_{i}$}{Perturbed matrix corresponding to SCC $C_i$, $\mathcal{L}_{i}=\mathcal{L}(C_{i})-\Delta_{i}$, where  $\Delta_i$ diagonal matrix corresponding to outgoing edges of SCC $C_i$, if $C_i$ is tSCC then $\Delta_i\equiv 0$}$\mathcal{L}_{i}=\mathcal{L}(C_{i})-\Delta_{i}$
for some\nomenclature{$\Delta$}{A diagonal matrix with non-negative entries}
$\Delta_{i}\not\equiv0$. In this case we call $\mathcal{L}_{i}$
a \emph{perturbed Laplacian matrix}, or simply a \emph{perturbed matrix}
and note the following property of $\mathcal{L}_{i}$ (proven in \cite{MirzaevGunawardena2013bmb}).

\begin{rem}\label{lem:perturbed} The Perturbed Laplacian matrix
of strongly connected graph $G$ is non-singular.\end{rem}

A \emph{directed spanning subgraph} of digraph $G$ is a connected
subgraph of $G$ that includes every vertex of $G$, so that any spanning
subgraph which is at the same time is a tree is called \emph{directed
spanning tree }(\nomenclature{DST}{Directed spanning tree}DST) of
the digraph $G$. We say that a DST, \nomenclature{$\mathscr{T}$}{Directed spanning tree}$\mathscr{T}$,
is \emph{rooted} at $i\in G$ if vertex $i$ is the only vertex in
$\mathscr{T}$ without any outgoing edges, and denote the set of DSTs
of digraph $G$ rooted at vertex $i$ with \nomenclature{$\Theta_i(G)$}{Set of DSTs of digraph $G$  rooted at vertex $i$}$\Theta_{i}(G)$.
Thus $\Theta_{i}(G)$ is a non-empty set of spanning trees for a strongly
connected digraph $G$. However, for an arbitrary digraph there maybe
no spanning tree rooted at specific vertex, in which case $\Theta_{i}(G)=\emptyset$.
In this case the corresponding element, $\mathcal{L}(G)_{(j)}$, is
zero, where \nomenclature{$A_{(ij)}$}{$ij$ -th minor of Laplacian matrix $A$  and is the determinant of $(n-1)\times (n-1)$  matrix that results from deleting row $i$  and column $j$}$\mathcal{L}(G)_{(ji)}$
denotes the $ji$-th minor of Laplacian matrix $\mathcal{L}(G)$ and
is the determinant of the $(n-1)\times(n-1)$ matrix that results
from deleting row $j$ and column $i$ of $\mathcal{L}(G)$.

Next we review the main theorem from \cite{MirzaevGunawardena2013bmb}
on which the results of this paper are based. The theorem utilizes
the digraph structure of digraph $G$ to calculate minors of a Laplacian.
The proof of this theorem can be found in several papers, and we direct
readers to \cite{MirzaevGunawardena2013bmb} for a proof with same
notations as in this article.%
\footnote{For more generalized versions of MTT such as all-minors Matrix-Tree
theorem and Matrix Forest Theorem we refer reader to \cite{Agaev2000,Chebotarev2002}. %
}
\begin{theorem}
\label{thm:(Matrix-Tree-Theorem)}\nomenclature{MTT}{Matrix-Tree Theorem}(Matrix-Tree
Theorem) If $G$ is digraph with $n$ vertices then the minors of
its Laplacian are given by 
\[
\mathcal{L}(G)_{(ij)}=(-1)^{n+i+j-1}\sum_{\mathscr{T}\in\Theta_{j}(G)}P_{\mathcal{\mathscr{T}}}\,,
\]
where $P_{\mathcal{\mathscr{T}}}$ is the product of all edge weights
in the spanning tree $\mathscr{T}$. 
\end{theorem}

An illustration of above theorem is depicted in Figure \ref{fig:Illustration-of-Algorithm},
where $\mathcal{L}(G)_{(23)}$ and $\mathcal{L}(G)_{32}$ minor of
$\mathcal{L}(G)$ are computed using spanning trees of digraph $G$.
Consequently, the MTT implies that the $ij$-th minor of the Laplacian
(up to sign) is the sum of all $P_{\mathcal{\mathscr{T}}}$ for each
spanning tree, $\mathscr{T}$, rooted at vertex $j$. Since all edges
of the digraph $G$ are non-negative numbers (zero only if there is
no such edge), then the expression $\rho_{i}^{G}=\sum_{\mathscr{T}\in\Theta_{i}(G)}P_{\mathscr{T}}$
will always be non-negative.%
\footnote{Later, we will define $\rho_{i}^{G}$ as an entry of the kernel element
of Laplacian.%
} If $G$ is strongly connected then $\Theta_{i}(G)\ne\emptyset$,
so $\rho_{i}^{G}$ is strictly positive. 

\begin{figure}
\centering \subfloat[A digraph $G$ with its set of spanning trees rooted at each of its
vertices. The root node is bolded in each spanning tree. ]{\begin{tikzpicture}[scale=0.8, auto=left, line width=1pt]  
\node[circle,draw=black!90,fill=blue!20] (1) at (0,0.7) {1};   
\node[circle,draw=black!90,fill=blue!20] (2) at (3,0.7)  {2};   
\node[circle,draw=black!90,fill=blue!20] (3) at (1.5,2.7)  {3};  
\node (4) at (2,-2.3) {$G$};

\begin{scope}[thick, ->,>=stealth', every node/.style = {above}, line width=1.5pt]  

\draw  (1) edge  node [below]{$a$} (2);
\draw  (2) edge  [bend right] node [right]{$b$} (3);
\draw  (3) edge  node [left][left]{$c$} (1);
\draw  (3) edge [bend right] node [left]{$d$} (2);

\end{scope}

\end{tikzpicture}
\hspace{2cm}
\begin{tikzpicture}[scale=0.8, auto=left, line width=1pt]  
%Spanning tree of 1%
\node[circle,draw=black!90,fill=blue!20,line width=2.2pt] (one1) at (0,0) {1};   
\node[circle,draw=black!90,fill=blue!20] (two1) at (3,0)  {2};  
\node[circle,draw=black!90,fill=blue!20] (three1) at (1.5,2)  {3};
\node (41) at (1.5,-3)  {$\Theta_{1}(G)$};

%Spanning tree of 2%

\node[circle,draw=black!90,fill=blue!20] (one2) at (5,2) {1};   
\node[circle,draw=black!90,fill=blue!20,line width=2.2pt] (two2) at (8,2)  {2};   
\node[circle,draw=black!90,fill=blue!20] (three2) at (6.5,4)  {3}; 

\node[circle,draw=black!90,fill=blue!20] (one3) at (5,-2) {1};   
\node[circle,draw=black!90,fill=blue!20,line width=2.2pt] (two3) at (8,-2)  {2};   
\node[circle,draw=black!90,fill=blue!20](three3) at (6.5,0)  {3};

\node (42) at (6.5,-3)  {$\Theta_{2}(G)$};

%Spanning tree of 3%
\node[circle,draw=black!90,fill=blue!20] (one4) at (10,0) {1};   
\node[circle,draw=black!90,fill=blue!20] (two4) at (13,0)  {2};   
\node[circle,draw=black!90,fill=blue!20, line width=2.2pt] (three4) at (11.5,2)  {3}; 

\node (43) at (11.5,-3)  {$\Theta_{3}(G)$};

\begin{scope}[thick, ->,>=stealth', every node/.style = {above}, line width=1.5pt]  
%Spanning tree of 1%
\draw  (two1) edge  node [left]{$b$} (three1);
\draw  (three1) edge  node [left][left]{$c$} (one1);

%Spanning tree of 2%

\draw  (one2) edge  node [above]{$a$} (two2);
\draw  (three2) edge  node [left]{$d$} (two2);

\draw  (three3) edge  node [left]{$c$} (one3);
\draw  (one3) edge  node [above]{$a$} (two3);

%Spanning tree of *%
\draw  (two4) edge  node [above]{$b$} (three4);
\draw  (one4) edge  node [above]{$a$} (two4);

\end{scope}

\end{tikzpicture}}

\subfloat[Associated Laplacian matrix of $G$. Two minors of the Laplacian matrix,
$\mathcal{L}(G)_{(23)}$ and $\mathcal{L}(G)_{(32)}$, are calculated
using the Matrix-Tree Theorem.]{%
\begin{minipage}[c][2.5\totalheight]{0.9\columnwidth}%
{\tikzstyle{every picture}+=[remember picture]  
$\mathcal{L}(G)=\left(\begin{array}{ccc} -a & 0 & c\\ a  & -b & \tikz[baseline]{\node[circle,draw=black!90,thick, anchor=base] (t1){$d$};}\\ 0 & \tikz[baseline]{\node[circle,draw=black!90,thick, anchor=base] (t2){$b$};} & -c-d \end{array}\right)$

\begin{tikzpicture}[overlay,>=latex] 
\node  (n1) at (10,2) {$\mathcal{L}\left(G\right)_{(23)}=\left|\begin{array}{cc} -a& 0\\  0 & b \end{array}\right|=-ab$}; 
\node  (n2) at (10,-1){$\mathcal{L}\left(G\right)_{(32)}=\left|\begin{array}{cc} -a& c\\  a & d \end{array}\right|=-ad-ac$}; 
\path[thick,->] (t1) edge  (n1.west); 
\path[thick,->] (t2) edge  (n2.west); 
\end{tikzpicture} }%
\end{minipage}

}

\caption{Illustration of Matrix-Tree Theorem\label{fig:Illustration-of-Algorithm}}
\end{figure}

Uno in his article \cite{Uno1996} provided an algorithm for enumerating
and listing all spanning trees of a general digraph and Ahsendorf
et al.~\cite{Ahsendorf2013} utilized Uno's algorithm to compute
minors of Laplacian matrix using the Matrix-Tree theorem. Having an
implementation of the MTT available, one can then calculate the kernel
elements of the Laplacian, $\mathcal{L}(G)$, using the following
two fairly well known Propositions (see \cite{MirzaevGunawardena2013bmb}
for proofs). 
\begin{proposition}
\label{prop:SCC kernel}If $G$ is strongly connected graph, then
$\ker\mathcal{L}(G)=span\left\{ \boldsymbol{\rho}^{G}\right\} $,
where \nomenclature{$\boldsymbol\rho^{G}$}{Kernel element of strongly connected digraph $G$ calculated by MTT}$\boldsymbol{\rho}^{G}$
is column vector with $\left(\boldsymbol{\rho}^{G}\right)_{i}=\rho_{i}^{G}>0$. 
\end{proposition}

Here, the kernel is defined in the conventional sense, $\ker\mathcal{L}(G)=\left\{ x\in\mathbb{R}^{n\times1}:\,\mathcal{L}(G)\cdot x=\mathbf{0}\right\} $.
Moreover, Proposition \ref{prop:SCC kernel} guarantees that a kernel
element has all positive elements, a fact which is not immediately
obvious using standard linear algebraic methods. When $G$ is not
a strongly connected digraph, the kernel elements of $\mathcal{L}(G)$
are constructed using kernel elements of its tSCCs. Specifically,
since for $i=1,\cdots,q$ each $\mathcal{L}_{p+i}=\mathcal{L}(C_{p+i})$
is a Laplacian matrix on its own, by Proposition \ref{prop:SCC kernel}
there exists \nomenclature{$\mathbb{R}^{m\times n}_{>0}$}{Set of all $m\times n$ matrices with striclty positve entiries}$\boldsymbol{\rho}^{C_{p+i}}\in\mathbb{R}_{>0}^{a_{p+i}\times1}$
such that $\mathcal{L}_{p+i}\cdot\boldsymbol{\rho}^{C_{p+i}}=\mathbf{0}$
and $a_{i}$ is the number of vertices in $C_{i}$. Then we can extend
this vector to \nomenclature{$\bar{\boldsymbol{\rho}}^{C_{i}}$}{A column vector, which is the extension of ${\rho}^{C_{i}}$, see Preliminaries section }$\bar{\boldsymbol{\rho}}^{C_{p+i}}\in\mathbb{R}_{>0}^{n\times1}$
by setting all entries with indices outside $C_{p+i}$ to zero:

\begin{equation}
\left(\bar{\boldsymbol{\rho}}^{C_{p+i}}\right)_{k}=\begin{cases}
\left(\boldsymbol{\rho}^{C_{p+i}}\right)_{k-m_{p+i-1}} & \text{if }m_{p+i-1}\le k\le m_{p+i}\\
0 & \text{otherwise}
\end{cases}\label{eq:Extension of rho}
\end{equation}

Since $\mathcal{L}(G)$ has lower-block diagonal structure and since
$\mathcal{L}_{p+i}\cdot\boldsymbol{\rho}^{C_{p+i}}=\mathbf{0}$ for
each $i=1,\cdots,q$ we have $\mathcal{L}(G)\cdot\bar{\boldsymbol{\rho}}^{C_{p+i}}=\mathbf{0}$.
This can be summarized in the following Proposition:
\begin{proposition}
\label{prop:General kernel}For any graph $G$,
\[
\ker\mathcal{L}(G)=span\left\{ \bar{\boldsymbol{\rho}}^{C_{p+1}},\dots,\bar{\boldsymbol{\rho}}^{C_{p+q}}\right\} \,,
\]

and $\dim\ker\mathcal{L}(G)=q$
\end{proposition}

To prove stability of the steady states we will use the following
theorem and corollary, which provides sufficiency conditions for the
solution to a dynamical system $d\mathbf{x}/dt=A\cdot\mathbf{x}$
coupled with initial condition $\mathbf{x}(0)=\mathbf{x}_{0}$ to
converge to a steady state (proof in \cite{MirzaevGunawardena2013bmb}).
Typically, the stability of a dynamics depends on the sign of the
real parts of the eigenvalues of $A$ as well as the algebraic and
geometric multiplicities of the zero eigenvalue. 
\begin{theorem}
\label{thm:Convergence}Suppose that the real matrix $A$ satisfies
following two conditions

1. If $\lambda$ is an eigenvalue of $A$, then either $\lambda=0$
or $Re(\lambda)<0$

2. $alg_{A}(0)=geo_{A}(0)$, where \nomenclature{$alg_A(0)$}{Algebraic multiplicity of zero eigenvalue of matrix $A$}$alg_{A}(0)$
and \nomenclature{$geo_A(0)$}{Geometric multiplicity of zero eigenvalue of matrix $A$}$geo_{A}(0)$
are the algebraic and geometric multiplicities of zero eigenvalue,
respectively. 

Then the solution of $d\mathbf{x}/dt=A\cdot\mathbf{x}$ converges
to a steady state as $t\to\infty$ for any initial condition. 
\end{theorem}

\begin{corollary}
\label{cor:alg and geo}The Laplacian of a weakly connected digraph
satisfies conditions of Theorem \ref{thm:Convergence}. Moreover,
$geo_{\mathcal{L}(G)}\left(0\right)=alg_{\mathcal{L}(G)}\left(0\right)=q$,
where $q$ is number of tSCCs of $G$. 
\end{corollary}

With these preliminary results in hand we provide stability analysis
for the SD dynamics as well as a graph theoretical algorithm for the
computation of steady states.

\section{\label{sec:Theoretical-Development}Theoretical Development}

In this section we will provide a thorough analysis of \nomenclature{SD dynamics}{Synthesis and degradation dynamics}
the synthesis and degradation dynamics (SD dynamics), (\ref{prop:General kernel}),
that we defined earlier. Suppose now that we add additional edges
to core digraph, $G$,
\[
\overset{s_{i}}{\longrightarrow}i\:,\, i\overset{d_{i}}{\longrightarrow}
\]
corresponding to zeroth-order synthesis and first-order degradation,
respectively. Each vertex can have any combination of synthesis and
degradation edges and the dynamics can now be described by the following
system of linear ordinary differential equations (ODEs): 
\begin{equation}
\frac{d\mathbf{x}}{dt}=\mathcal{L}(G)\cdot\mathbf{x}-D\cdot\mathbf{x}+\mathbf{s}\,.\label{eq:3}
\end{equation}
Here $\mathcal{L}(G)$ is the Laplacian matrix of the core digraph
$G$, $D$ is a diagonal matrix with $\left(D\right)_{ii}=d_{i}$,
and $\mathbf{s}$ is a column vector with $\left(\mathbf{s}\right)_{i}=s_{i}$,
using the convention that $d_{i}$ or $s_{i}$ is zero if the corresponding
partial edge at vertex $i$ is absent. 

The presence of synthesis without degradation yields unstable dynamics.
Therefore, whenever we have $D\equiv\mathbf{0}$ we assume that $\mathbf{s}\equiv\mathbf{0}$.
In this case the system reduces to Laplacian dynamics, for which a
thorough analysis was given in \cite{MirzaevGunawardena2013bmb}.
From now on we will assume that at least one element of $D$ is nonzero.
\begin{proposition}
The dynamics defined by Equation (\ref{eq:3}) have a unique solution
for a given initial condition. \end{proposition}
\begin{proof}
This can easily be verified as the right hand side of the dynamics,
$f(\mathbf{x})=\mathcal{L}(G)\cdot\mathbf{x}-D\cdot\mathbf{x}+\mathbf{s}$
is affine in $\mathbf{x}$ and thus also \emph{Lipschitz }continuous.
Thus the existence of unique continuous solution is guaranteed. 
\end{proof}

The next question to be answered is to identify the conditions under
which the SD dynamics possess steady state solution(s). In order to
derive the necessary and sufficient conditions for existence of steady
state solution we define a \emph{complementary digraph}, \nomenclature{$G^{\star}$}{Complementary digraph of $G$, which is formed by directing all synthesis and degradation edges to new vertex $*$}$G^{\star}$,
which is formed by defining new vertex, $\star$, such that 
\[
\star\ce{->[s_{i}]}i\:\mbox{ or }\, i\ce{->[d_{i}]}\star
\]

For the sake of simplicity we will divide our results into two cases:
when $G^{\star}$ is a strongly connected digraph and when it is not
strongly connected.

\subsection{\label{sub:SC framework}Strongly connected case}

Assume that complementary digraph $G^{\star}$ is strongly connected.
Let $F$ denote the matrix of the Laplacian minus the degradation
matrix $D$: 

\begin{equation}
F=\mathcal{L}(G)-D=\left(\begin{array}{ccc|ccc}
\boxed{\mathcal{L}_{1}-D_{1}} & \cdots & \mathbf{0} & \mathbf{0} & \cdots & \mathbf{0}\\
\vdots & \ddots & \vdots & \vdots & \ddots & \vdots\\*
+ & \cdots & \boxed{\mathcal{L}_{p}-D_{p}} & \mathbf{0} & \cdots & \mathbf{0}\\
\hline + & \cdots & + & \boxed{\mathcal{L}_{p+1}-D_{p+1}} & \mathbf{0} & \mathbf{0}\\
\vdots & \ddots & \vdots & \vdots & \ddots & \vdots\\*
+ & \cdots & + & \mathbf{0} & \mathbf{0} & \boxed{\mathcal{L}_{p+q}-D_{p+q}}
\end{array}\right)\label{eq:partition}
\end{equation}
Suppose that there is some index $m\in\{p+1,\dots,p+q\}$ for which
$D_{m}\equiv0$. Then $\mathcal{L}_{m}-D_{m}=\mathcal{L}_{m}=\mathcal{L}(C_{m})$,
which implies that $C_{m}$ is preserved as a tSCC in the complementary
digraph $G^{\star}$. In this case, the vertices corresponding to
the tSCC $C_{m}$ cannot be reached from any other SCC of $G^{\star}$,
which in turn contradicts the fact that graph $G^{\star}$ is strongly
connected. Thus each matrix $\mathcal{L}_{i}-D_{i}$ is a perturbed
Laplacian matrix of some strongly connected digraph, so that from
Remark \ref{lem:perturbed} each $\mathcal{L}_{i}-D_{i}$ is a non-singular
matrix for $i=p+1,\cdots,p+q$. For $i=1,\cdots,p$ as each of the
$\mathcal{L}_{i}$'s is already a \emph{perturbed} matrix, ``perturb''-ing
them further does not change the fact that they are non-singular.
Thus the matrix $F$ is a non-singular matrix, since its diagonal
components are all non-singular matrices and the unique steady state
solution is given algebraically as

\nomenclature{$\mathbf{x}_{\scriptstyle{ES}}$}{Steady state solution}
\begin{equation}
\mathbf{x}_{{\scriptscriptstyle ES}}=-(\mathcal{L}(G)-D)^{-1}\cdot\mathbf{s}\,.\label{eq:4}
\end{equation}
 However, we are interested in computing the ES by means of graph
theory. Towards this end, we define a change of variables 
\[
\mathbf{x}=\mathbf{y}-(\mathcal{L}(G)-D)^{-1}\cdot\mathbf{s}\,.
\]
This substitution transforms the original SD dynamics into 
\begin{equation}
\frac{d\mathbf{y}}{dt}=(\mathcal{L}(G)-D)\cdot\mathbf{y}\,.\label{eq:2}
\end{equation}
From Proposiotion \ref{cor:alg and geo} we know that eigenvalues
of the Laplacian\emph{, }$\mathcal{L}(G)$,\emph{ }satisfy $Re(\lambda)\le0$,
where equality holds if and only if $\lambda=0$. The proof of this
result follows from applying the \emph{Gershgorin theorem }to the
columns of the matrix $\mathcal{L}(G)$. Then the matrix $D$ in $\mathcal{L}(G)-D$
shifts the centers of the Gershgorin discs further to left on the
real line without changing their radii, so we will still have $Re(\lambda)\le0$
for eigenvalues of the matrix $\mathcal{L}(G)-D$. On the other hand,
the matrix $\mathcal{L}(G)-D$ is non-singular, so it follows that
$Re(\lambda)<0$. This result in turn implies that solution of the
system given in (\ref{eq:2}) converges to the trivial steady state,
$\mathbf{y}_{{\scriptscriptstyle ES}}=0$. Thus we can now state the
following theorem.
\begin{theorem}
\label{thm:Given-a-strongly}Given a strongly connected digraph $G$,
the SD dynamics (\ref{eq:3}) have a unique stable steady state solution.
\end{theorem}

The symbolic computation of the ES solution using (\ref{eq:4}) can
be very expensive even for small number of vertices. Therefore we
restate an algorithm, given in \cite{Gunawardena2012}, which uses
graphical structure of graph $G^{\star}$ to calculate steady state
solution given in (\ref{eq:4}). 

Let $\mathbbm{1}=(1,\cdots,1)^{T}$ be a vector of all ones. At steady
state we have $\frac{d\mathbf{x}}{dt}=\mathbf{0}$ and using the fact
that $\mathbbm{1}^{T}\cdot\mathcal{L}(G)=\mathbf{0}$ it follows from
(\ref{eq:New dynamics}) that 
\begin{equation}
d_{1}x_{1}+\dots+d_{n}x_{n}=s_{1}+\dots+s_{n}\,.\label{eq:balance}
\end{equation}
In other words at steady state we should have an overall balance in
synthesis and degradation. The Laplacian $\mathcal{L}(G^{\star})$
of the digraph $G^{\star}$ can then be related to the Laplacian $\mathcal{L}(G)$
of the digraph $G$ 

\begin{equation}
\mathcal{L}(G^{\star})=\left(\begin{array}{c|c}
\mathcal{L}(G) & \mathbf{0}\\
\hline \mathbf{0} & \mathbf{0}
\end{array}\right)+\left(\begin{array}{c|c}
-D & \mathbf{s}\\
\hline \mathbbm{1}^{T}\cdot D & -\mathbbm{1}^{T}\cdot\mathbf{s}
\end{array}\right)\label{eq:5}
\end{equation}

Suppose now that we have overall balance in synthesis and degradation
then using (\ref{eq:5}) it is easy to see that $(x_{1},\cdots,x_{n},1)$
is a steady state of 
\begin{equation}
\frac{d\mathbf{x}}{dt}=\mathcal{L}(G^{\star})\cdot\mathbf{x}\label{eq:G plus system}
\end{equation}
 if and only if $(x_{1},\cdots,x_{n})$ is a steady state of SD dynamics
given in (\ref{eq:3}). Since $G^{\star}$ is strongly connected,
the MTT provides a basis element for the kernel of the Laplacian matrix
$\mathcal{L}(G^{\star})$, $\ker\left\{ \mathcal{L}(G^{\star})\right\} =\text{span}\{\boldsymbol{\rho}^{G^{\star}}\}$
\cite{Gunawardena2012}. Consequently, the unique steady state $\mathbf{x}_{{\scriptscriptstyle ES}}$
is given by

\begin{equation}
\left(\mathbf{x}_{{\scriptscriptstyle ES}}\right)_{i}=\frac{\left(\boldsymbol{\rho}^{G^{\star}}\right)_{i}}{\left(\boldsymbol{\rho}^{G^{\star}}\right)_{\star}}\label{eq: SC steady state MTT}
\end{equation}

Since any steady state solution of (\ref{eq:G plus system}) can be
written as scalar multiple of kernel element, $\boldsymbol{\rho}^{G^{\star}}$,
that single degree of freedom is used to guarantee that $\left(\mathbf{x}_{{\scriptscriptstyle ES}}\right)_{\star}=1$
(synthesis and degradation vertex, $\star$). This condition also
ensures that overall balance in synthesis and degradation (\ref{eq:balance})
is satisfied.

\subsection{General Case}

In contrast to the strongly-connected case, the steady state solutions
do not always exist in the general case. First, we will derive conditions
to assure the existence of a steady state solution. Then, we will
show that provided we have a steady state solution $\mathbf{x}_{{\scriptscriptstyle ES}}$,
the system converges to this $\mathbf{x}_{{\scriptscriptstyle ES}}$
as $t\to\infty$. Third, we provide a framework for construction of
$\mathbf{x}_{{\scriptscriptstyle ES}}$ using the underlying graph
structure of graph $G$ with illustration of results using a hypothetical
example. 

In the case that the digraph $G^{\star}$ is not strongly connected,
in the partition of the matrix $F$ (\ref{eq:partition}) there is
at least one $i\in\left\{ p+1,\dots,p+q\right\} $ such that $D_{i}\equiv\mathbf{0}$.
Let $\left\{ i_{1},\cdots,i_{k}\right\} \subseteq\left\{ p+1,\cdots,p+q\right\} $
be a set for which $D_{i_{1}}=\dots=D_{i_{k}}\equiv\mathbf{0}$, then
we can relabel the vertices corresponding to the tSCC\emph{ }such
that matrices $\mathcal{L}_{i_{1}},\dots,\mathcal{L}_{i_{k}}$ are
positioned in the lower right of matrix $F$, 
\begin{align*}
F=\mathcal{L}(G)-D & =\left(\begin{array}{ccc|ccc}
\boxed{\mathcal{L}_{1}-D_{1}} & \cdots & \mathbf{0} & \mathbf{0} & \cdots & \mathbf{0}\\
\vdots & \ddots & \vdots\\*
+ & + & \boxed{\mathcal{L}_{p+q-k}-D_{p+q-k}} & \mathbf{0} & \cdots & \mathbf{0}\\
\hline + & \cdots & + & \boxed{\mathcal{L}_{p+q-k+1}} & \cdots & \mathbf{0}\\
\vdots & \ddots & \vdots & \vdots & \ddots & \vdots\\*
+ & \cdots & + & \mathbf{0} & \cdots & \boxed{\mathcal{L}_{p+q}}
\end{array}\right)
\end{align*}

\begin{equation}
=\left(\begin{array}{ccc|ccc}
\boxed{\mathcal{M}_{1}} & \cdots & \mathbf{0} & \mathbf{0} & \cdots & \mathbf{0}\\
\vdots & \ddots & \vdots & \vdots & \ddots & \vdots\\*
+ & + & \boxed{\mathcal{M}_{r}} & \mathbf{0} & \cdots & \mathbf{0}\\
\hline + & \cdots & + & \boxed{\mathcal{L}_{r+1}} & \mathbf{0} & \mathbf{0}\\
\vdots & \ddots & \vdots & \vdots & \ddots & \vdots\\*
+ & \cdots & + & \mathbf{0} & \mathbf{0} & \boxed{\mathcal{L}_{r+k}}
\end{array}\right)=\left(\begin{array}{c|c}
N & \mathbf{0}\\
\hline B & T
\end{array}\right)\label{eq:General case Decomposition}
\end{equation}
where \emph{$r=p+q-k$}, each $\mathcal{M}_{i}=\mathcal{L}_{i}-D_{i}=\mathcal{L}(C_{i})-\Delta_{i}-D_{i}$
is a \emph{perturbed} Laplacian matrix of some SCC $C_{i}$\emph{
}and each $\mathcal{L}_{i}$ corresponds to the Laplacian matrix of
some tSCC\emph{ }in graph $G^{\star}$\emph{.} This relabeling is
always possible because the labeling procedure described in Section
\ref{sec:Preliminary-results} does not provide any restriction on
individual labeling of vertices located in the set of tSCCs.

Next we present a theorem, which provides the necessary and sufficient
conditions in order for an ES to exist. For that we partition the
synthesis vector $\mathbf{s}$ such that it matches up with the partition
of the matrix $F$, 

\[
F=\left(\begin{array}{c|c}
N & \mathbf{0}\\
\hline B & T
\end{array}\right),\qquad\mathbf{s}=\left(\begin{array}{c}
\mathbf{s}'\\
\hline \mathbf{s}''
\end{array}\right)\,.
\]

\begin{theorem}
\label{thm:BNS}When $G^{\star}$ is not strongly connected, the necessary
and sufficient conditions for existence of an ES solution are\end{theorem}
\begin{enumerate}
\item $\mathbf{s}''\equiv\mathbf{0}$
\item $B\cdot N^{-1}\cdot\mathbf{s}'\equiv\mathbf{0}$\end{enumerate}
\begin{proof}
Let us first derive the equivalent statements for the existence of
an ES solution. Finding a steady state solution of the system is equivalent
to solving the linear system 
\begin{equation}
\left(\mathcal{L}(G)-D\right)\cdot\mathbf{x}=-\mathbf{s}\,.\label{eq:steady state equivalent system}
\end{equation}
Thus a steady state solution exists if and only if $-\mathbf{s}\in\text{Range}\left\{ \mathcal{L}(G)-D\right\} $.
Let us apply simple row reduction (i.e., Gaussian elimination) to
the augmented matrix $(\begin{array}{c:c}\mathcal{L}(G)-D & \mathbf{s}\end{array})$

\begin{align*} 
&\left(\begin{array}{ccc|ccc:c} 
\boxed{\mathcal{M}_{1}} & \cdots & \mathbf{0} & \mathbf{0} & \cdots & \mathbf{0}\\
\vdots & \ddots & \vdots & \vdots & \ddots & \vdots & -\mathbf{s}'
\\* * & \cdots & \boxed{\mathcal{M}_{r}} & \mathbf{0} & \cdots & \mathbf{0}\\ 
\cmidrule[0.5pt]{1-6} \cmidrule[1pt]{7-7} & B_{1} &  & \boxed{\mathcal{L}_{r+1}} & \mathbf{0} & \mathbf{0} & -\mathbf{s}^{(1)}\\  
& \vdots &  & \vdots & \ddots & \vdots&\vdots\\  & B_{k} &  & \mathbf{0} & \mathbf{0} & \boxed{\mathcal{L}_{r+k}} & -\mathbf{s}^{(k)} \end{array}\right) \\
&\\
\longrightarrow 
& \left(\begin{array}{ccc|ccc:c} 
\boxed{\mathbb{I}_{a_{1}}} & \cdots & \mathbf{0} & \mathbf{0} & \cdots & \mathbf{0}\\ 
\vdots & \ddots & \vdots & \vdots & \ddots & \vdots & -N^{-1}\mathbf{s}'\\ 
\mathbf{0} & \mathbf{0} & \boxed{\mathbb{I}_{a_{r}}} & \mathbf{0} & \cdots & \mathbf{0}\\ 
\cmidrule[0.5pt]{1-6} \cmidrule[1pt]{7-7} & B_{1} &  & \boxed{\mathcal{L}_{r+1}} & \mathbf{0} & \mathbf{0} & -\mathbf{s}^{(1)}\\  
& \vdots &  & \vdots & \ddots & \vdots &\vdots\\  
& B_{k} &  & \mathbf{0} & \mathbf{0} & \boxed{\mathcal{L}_{r+k}} & -\mathbf{s}^{(k)} \end{array}\right)\\ 
&\\
\longrightarrow 
&\left(\begin{array}{ccc|ccc:c} 
\boxed{\mathbb{I}_{a_{1}}} & \cdots & \mathbf{0} & \mathbf{0} & \cdots & \mathbf{0}\\ 
\vdots & \ddots & \vdots & \vdots & \ddots & \vdots & -N^{-1}\mathbf{s}'\\
\mathbf{0} & \mathbf{0} & \boxed{\mathbb{I}_{a_{r}}} & \mathbf{0} & \cdots & \mathbf{0}\\ 
\cmidrule[0.5pt]{1-6} \cmidrule[1pt]{7-7} \mathbf{0} & \cdots & \mathbf{0} & \boxed{\mathcal{L}_{r+1}} & \mathbf{0} & \mathbf{0} & -\mathbf{s}^{(1)}+B_{1}N^{-1}\mathbf{s}'\\ 
\vdots & \ddots & \vdots & \vdots & \ddots & \vdots & \vdots\\ 
\mathbf{0} & \cdots & \mathbf{0} & \mathbf{0} & \mathbf{0} & \boxed{\mathcal{L}_{r+k}} & -\mathbf{s}^{(k)}+B_{k}N^{-1}\mathbf{s}' \end{array}\right)\ . 
\end{align*}

So the above system (\ref{eq:steady state equivalent system}) has
a solution if and only if each partial linear system 
\begin{equation}
\mathcal{L}_{r+i}\cdot\mathbf{z}^{(i)}=-\mathbf{s}^{(i)}+B_{i}\cdot N^{-1}\cdot\mathbf{s}'\label{eq:Eq Condition with L_{r+i}}
\end{equation}
 has a solution. Equation (\ref{eq:Eq Condition with L_{r+i}}) provides
an equivalent condition for the existence of an ES solution of the
SD dynamics. At this point we will prove that (\ref{eq:Eq Condition with L_{r+i}})
is satisfied if and only if two conditions of the theorem are satisfied. 

Let us assume that (\ref{eq:Eq Condition with L_{r+i}}) holds true.
Then each partial linear system has a solution if the following condition
is satisfied:
\begin{equation}
\mathbbm{1}^{T}\cdot\mathcal{L}_{r+i}\cdot\mathbf{z}^{(i)}=\mathbbm{1}^{T}\cdot\mathcal{L}(C_{r+i})\cdot\mathbf{z}^{(i)}=\mathbf{0}=-\mathbbm{1}^{T}\cdot\mathbf{s}^{(i)}+\mathbbm{1}^{T}\cdot B_{i}\cdot N^{-1}\cdot\mathbf{s}'\qquad i\in\left\{ 1,\cdots,k\right\} \,.\label{eq:1Lz}
\end{equation}
To proceed further we need the nontrivial fact that all the entries
of the matrix $N^{-1}$ are non-positive real numbers. For that reason,
we have devoted all of Section \ref{sec: negativity of inverse} for
the proof of this fact as well as presenting a graph theoretical algorithm
for computation of $N^{-1}$. All entries of the vector $\mathbf{s}^{(i)}$
and matrix $B_{i}$ are non-negative real numbers, because all edge
weights are non-negative real numbers by definition. Therefore, each
of the products $B_{i}\cdot N^{-1}\cdot\mathbf{s}'$ are the matrices
with non-positive entries. This in turn implies that both of the summands
in (\ref{eq:1Lz}) are equal to zero, 
\[
-\mathbbm{1}^{T}\cdot\mathbf{s}^{(i)}\equiv\mathbf{0}\textrm{ and }\mathbbm{1}^{T}\cdot B_{i}\cdot N^{-1}\cdot\mathbf{s}'\equiv\mathbf{0}\qquad i\in\left\{ 1,\cdots,k\right\} 
\]
Recall that a sum of non-negative ($\mathbb{R}_{\ge0}$) real numbers
is equal to zero if and only if each of the numbers are equal to zero.
Hence, 
\begin{equation}
\mathbf{s}^{(i)}\equiv\mathbf{0}\textrm{ and }B_{i}\cdot N^{-1}\cdot\mathbf{s}'\equiv\mathbf{0}\qquad i\in\left\{ 1,\cdots,k\right\} \label{eq:two conditions partioned}
\end{equation}
which is equivalent to the two conditions of the theorem,

\begin{equation}
\mathbf{s}''=\left(\mathbf{s}_{1},\cdots,\mathbf{s}_{k}\right)^{T}\equiv\mathbf{0}\textrm{ and }B\cdot N^{-1}\cdot\mathbf{s}'\equiv\mathbf{0}\,.\label{eq: 1}
\end{equation}

Conversely, assume that two conditions of the theorem are satisfied.
Then it easy to observe that (\ref{eq:two conditions partioned})
also holds true. Consequently, it follows that the linear system  (\ref{eq:Eq Condition with L_{r+i}})
reduces to
\begin{equation}
\mathcal{L}_{r+i}\cdot\mathbf{z}^{(i)}=\mathcal{L}(C_{r+i})\cdot\mathbf{z}^{(i)}=-\mathbf{s}_{i}+B_{i}\cdot N^{-1}\cdot\mathbf{s}'=\mathbf{0}\label{eq:Reduces to LD}
\end{equation}
which always has a solution. Moreover, the solution of the above linear
system (\ref{eq:Reduces to LD}) can be constructed graph theoretically
by Proposition \ref{prop:SCC kernel}. 
\end{proof}

The first condition of the theorem, $\mathbf{s}''\equiv\mathbf{0}$,
can be interpreted as follows: a necessary condition for the existence
of a steady state is that if a tSCC\emph{ }does not have degradation
edge, it should also not have a synthesis edge. On the other hand,
one can also visualize these conditions in terms of chemical reactions.
If there is continuous inflow of substrates into the production part
of the reaction and a lack of outflow, then reaction will grow without
bound. The second condition of the theorem identifies nodes without
degradation, which also contribute directly (or indirectly) to tSCCs.
Indeed, this type of nodes also cause the SD dynamics to grow without
bound. 

Next we show that (\ref{eq:New dynamics}) can be transformed into
homogeneous system of linear differential equations, provided that
both conditions of Theorem \ref{thm:BNS} are fulfilled. As in the
previous section we denote matrix $L(G)-D$ with $F$, and partition
$F$ and $\mathbf{s}$ as, 

\[
F=\left(\begin{array}{c|c}
N & \mathbf{0}\\
\hline B & T
\end{array}\right)\qquad\mathbf{s}=\begin{pmatrix}\mathbf{s}'\\
\hline \mathbf{0}
\end{pmatrix}
\]
where $F\in\mathbb{R}_{\ge0}^{n\times n},\, N\in\mathbb{R}_{\ge0}^{m\times m},\, B\in\mathbb{R}_{\ge0}^{m\times(n-m)},\, T\in\mathbb{R}_{\ge0}^{(n-m)\times(n-m)}$
and $\mathbf{s}\in\mathbb{R}_{\ge0}^{n\times1},\,\mathbf{s}'\in\mathbb{R}_{\ge0}^{m\times1}$.
Let us also define a matrix $Q\in\mathbb{R}_{\le0}^{n\times n}$ 
\[
Q=\left(\begin{array}{c|c}
N^{-1} & \mathbf{0}\\
\hline \mathbf{0} & \mathbf{0}
\end{array}\right)\,,
\]
to be used in the change of variable $\mathbf{x}=\mathbf{y}-Q\cdot\mathbf{s}$.
This substitution transforms (\ref{eq:New dynamics}) into 
\[
\frac{d\mathbf{x}}{dt}=\frac{d\mathbf{y}}{dt}=F\cdot\mathbf{y}-F\cdot Q\cdot\mathbf{s}+\mathbf{s}=F\cdot\mathbf{y}-\begin{pmatrix}\mathbf{s}'\\
\hline B\cdot N^{-1}\cdot\mathbf{s}'
\end{pmatrix}+\begin{pmatrix}\mathbf{s}'\\
\hline \mathbf{0}
\end{pmatrix}\,.
\]
Assuming that an ES solution of the SD dynamics exists, then by Theorem
\ref{thm:BNS} we have that $B\cdot N^{-1}\cdot\mathbf{s}'\equiv\mathbf{0}$,
from which it follows that 
\begin{equation}
\frac{d\mathbf{y}}{dt}=F\cdot\mathbf{y}\,.\label{eq:System with matrix F}
\end{equation}

\begin{theorem}
For any given initial condition \label{lem:System--converges} the
dynamics defined in (\ref{eq:System with matrix F}) converges to
a unique steady state as $t\to\infty$.\end{theorem}
\begin{proof}
We will prove this theorem\emph{ }by showing that matrix $F$ satisfies
both conditions of Theorem \ref{thm:Convergence}. First note that
by definition, the matrix $F$ is a Laplacian matrix minus a non-negative
diagonal matrix, $F=\mathcal{L}(G)-D$. Hence, it follows that 
\[
\sum_{v\ne i}|\left(F\right)_{vi}|=\sum_{v\ne i}\left(\mathcal{L}(G)\right)_{vi}=|\left(\mathcal{L}(G)\right)_{ii}|\le|\left(\mathcal{L}(G)\right)_{ii}|+d_{i}=\left|\left(F\right)_{ii}\right|
\]
Therefore, if we apply Gerschgorin's theorem to the columns of the
matrix $F$, we see that each eigenvalue of $F$ is located in the
discs of the form 
\[
\left\{ z\in\mathbf{C}\ |\ \left|z+|\left(\mathcal{L}(G)\right)_{ii}|+d_{i}\right|\le|(\mathcal{L}(G))_{ii}|\right\} \,.
\]
A disc touches the $y-axis$ from the left hand side if and only if
$|\left(\mathcal{L}(G)\right)_{ii}|+d_{i}=|\left(\mathcal{L}(G)\right)_{ii}|$,
or $d_{i}=0$. Hence for an eigenvalue, $\lambda$, of the matrix
$F$ we conclude that $Re(\lambda)\le0$, where equality holds if
and only if $\lambda=0$. Thus, the matrix $F$ satisfies first condition
of Theorem \ref{thm:Convergence}. 

On the other hand, from the lower-block diagonal structure of the
matrix $F$ and Corollary \ref{cor:alg and geo} , it follows that
\begin{align*}
geo_{F}(0) & =\dim\left\{ \ker F\right\} =\dim\left\{ \ker N\right\} +\dim\left\{ \ker\mathcal{L}_{r+1}\right\} +\cdots+\dim\left\{ \ker\mathcal{L}_{r+k}\right\} \\
 & =0+geo_{\mathcal{L}_{r+1}}(0)+\cdots+geo_{\mathcal{L}_{r+k}}(0)=k\,.
\end{align*}
Remember that each matrix $\mathcal{L}_{r+i}$ is the Laplacian matrix
of the tSCC $C_{r+i}$,\emph{ }so the $\dim\ker\left\{ \mathcal{L}_{r+i}\right\} =geo_{\mathcal{L}_{r+i}}(0)=1=alg_{\mathcal{L}_{r+i}}(0)$.\emph{
}Again the block diagonal structure of $F$ suggests that
\begin{align*}
alg_{F}(0) & =alg_{\mathcal{M}_{1}}(0)+\cdots+alg_{\mathcal{M}_{r}}(0)+alg_{\mathcal{L}_{r+1}}(0)+\cdots+alg_{\mathcal{L}_{r+k}}(0)\\
 & =0+\cdots+0+1+\cdots+1=k\,.
\end{align*}
 Since each matrix $\mathcal{M}_{i}$ is non-singular, none of their
eigenvalues are zero. Hence, it follows that the dynamics (\ref{eq:System with matrix F})
satisfy the second condition of Theorem \ref{thm:Convergence}, $alg_{F}(0)=geo_{F}(0)=k$.
Therefore, the matrix $F$ satisfies both conditions of Theorem\emph{
}\ref{thm:Convergence}, which in turn implies that the dynamics defined
in (\ref{eq:System with matrix F}) converge to a unique ES for any
given initial condition. 
\end{proof}

Now we will provide a framework for finding the ES solution of the
SD dynamics. Define $R$ as an $n\times k$ matrix whose columns are
a basis elements of the column null space (right kernel) of the matrix
$F$.%
\footnote{Recall that dimensions of row and column null spaces of a matrix are
same. In fact, from Corollary \ref{cor:alg and geo} we have this
dimension equal to number of tSCC of digraph $G^{\star}$. %
}\emph{ }Analogously, define $L$ such that is a $k\times n$ matrix
whose rows are basis elements of the row null space (left kernel)
of the matrix $F$. Then these matrices satisfy
\[
F\cdot R=\mathbf{0}\text{ and }L\cdot F=\mathbf{0}
\]
 Naturally, $L$ and $R$ can be chosen so that the following equation
hold

\begin{equation}
L\cdot R=\mathbb{I}_{k}\,.\label{eq:LR}
\end{equation}
 Since matrices $L$ and $R$ are not uniquely defined, equation (\ref{eq:LR})
serves as a normalization condition. In fact, in the subsequent section
we discuss an example of such normalization. The following lemma gives
a representation of the ES of the dynamics defined in (\ref{eq:System with matrix F})
in terms of matrices $R$ and $L$. 
\begin{lemma}
\label{lem:steady state solution}Assume that $F$ is a matrix for
which  (\ref{eq:LR}) holds and when coupled with the initial condition
$\mathbf{y}(0)=\mathbf{y}_{0}$ the solution of system (\ref{eq:System with matrix F})
converges to a steady state $\mathbf{y}_{{\scriptscriptstyle ES}}$
as $t\to\infty$. Then $\mathbf{y}_{{\scriptscriptstyle ES}}=R\cdot L\cdot\mathbf{y}_{0}$. \end{lemma}
\begin{proof}
The solution of the dynamics (\ref{eq:System with matrix F}) can
be given in exponential form as $\mathbf{y}(t)=\exp(Ft)\cdot\mathbf{y}_{0}$,
so\nomenclature{$\mathbb{I}_{n}$}{$n\times n$ identity matrix}
\begin{equation}
\mathbf{y}(t)=\left(\mathbb{I}_{n}+\left(Ft\right)+\frac{(Ft)^{2}}{2!}+\cdots\right)\cdot\mathbf{y}_{0}=\left(\mathbb{I}_{n}+F\cdot A(t)\right)\cdot\mathbf{y}_{0},\label{eq:exponential}
\end{equation}
where $A(t)$ is some matrix with time defined functions as entries.
From (\ref{eq:LR}) and (\ref{eq:exponential}) it follows that $L\cdot\mathbf{y}(t)=L\mathbf{y}_{0}+L\cdot F\cdot A(t)\mathbf{y}_{0}=L\mathbf{y}_{0}$.
Therefore, asymptotically as $t\to\infty$ we find that $L\cdot\mathbf{y}_{{\scriptscriptstyle ES}}=L\cdot\mathbf{y}(0)$.
On the other hand steady state \textbf{$\mathbf{y}_{{\scriptscriptstyle ES}}$},
should also satisfy $\frac{d\mathbf{{\scriptstyle \mathbf{y}}}_{{\scriptscriptstyle {\scriptscriptstyle ES}}}}{dt}=\mathbf{0}=F\cdot\mathbf{y}_{{\scriptscriptstyle ES}}$,
then vector $\mathbf{y}_{{\scriptscriptstyle ES}}$ is element of
the column null space of the matrix $F$. In other words, the vector
$\mathbf{y}_{{\scriptscriptstyle ES}}$ can be written as linear combination
of the elements of the column null space, so there exist some vector
\nomenclature{$\mathbb{R}^{m\times n}$}{The set of all $m\times n$ matrices with real entries.}$\mathbf{d}\in\mathbb{R}^{n\times1}$
such that $\mathbf{y}_{{\scriptscriptstyle ES}}=R\cdot\mathbf{d}$
. Therefore, $L\cdot\mathbf{y}(0)=L\cdot\mathbf{y}_{{\scriptscriptstyle ES}}=L\cdot\left(R\cdot\mathbf{d}\right)=\mathbf{d}$,
so that $\mathbf{y}_{{\scriptscriptstyle ES}}=R\cdot\mathbf{d}=R\cdot L\cdot\mathbf{y}(0)$,
as desired. 
\end{proof}

Once the steady state solution of the dynamics (\ref{eq:System with matrix F})
is found, the steady state solution of the SD dynamics (\ref{eq:3})
can be found using back substitution: 
\begin{equation}
\mathbf{x}_{{\scriptscriptstyle ES}}=\mathbf{y}_{{\scriptscriptstyle ES}}-Q\cdot\mathbf{s}=R\cdot L\cdot\mathbf{y}_{0}-Q\cdot\mathbf{s}=R\cdot L\cdot\mathbf{x}_{0}+\left(R\cdot L-\mathbb{I}_{n}\right)\cdot Q\cdot\mathbf{s}\,.\label{eq:SS solution}
\end{equation}

\subsubsection{\label{sub:Construction-of-matrices}Construction of the matrices
$R$ and $L$}

We next discuss the graph theoretical procedure to construct matrices
$R$ and $L$ that satisfy (\ref{eq:LR}). The general strategy is
to calculate matrix $R$ using Proposition \ref{prop:General kernel},
then construct the uniquely defined matrix $L$ that satisfies (\ref{eq:LR}).
Consider the block decomposition of matrix $F$ given in (\ref{eq:General case Decomposition}),
decompose matrices $R$ and $L$ such that\setlength\BAextraheightafterhline{5pt} \setlength\tabcolsep{0pt}
\begin{equation}
 F= \begin{blockarray}{ccc}  \scriptstyle{n-u} &\scriptstyle{u} &  \\ \begin{block}{\Left{}{(\;}c|c<{\;})c}   N&\mathbf{0}&\scriptstyle{n-u}\\  \BAhhline{--||&} B&T&\scriptstyle{u}\\ \end{block}  \end{blockarray},\qquad L=\begin{blockarray}{ccc}   \scriptstyle{n-u} & \scriptstyle{u}&  \\ \begin{block}{\Left{}{(\;}c|c<{\;})c}   X&U&\scriptstyle{k}\\ \end{block}  \end{blockarray}, \qquad R=\begin{blockarray}{cc}     \scriptstyle{k}&  \\ \begin{block}{\Left{}{(\;}c<{\;})c}   Y&\scriptstyle{n-u}\\  \BAhhline{-||&} V&\scriptstyle{u}\\ \end{block}  \end{blockarray}
\label{LR partition}
\end{equation}where $k$ is the number of tSCCs of complementary digraph $G^{\star}$,
$u$ is number of vertices that are in tSCCs $C_{r+1},\dots,C_{r+k}$
, $X$ is $k\times(n-u)$, $U$ is $k\times u$, $N$ is $(n-u)\times(n-u)$,
$B$ is $u\times(n-u)$, $T$ is $u\times u$, $Y$ is $(n-u)\times k$
and $V$ is $u\times k$. 

Consequently, the kernel elements of the matrix $F$ are constructed
using Proposition \ref{prop:General kernel}, $\ker\left\{ F\right\} =span\left\{ \bar{\boldsymbol{\rho}}^{C_{r+1}},\cdots,\bar{\boldsymbol{\rho}}^{C_{r+k}}\right\} $
using the tSCCs, $\left\{ C_{r+1},\cdots,C_{r+k}\right\} $ and thus
we have $F\cdot\bar{\boldsymbol{\rho}}^{C_{r+i}}=0,\quad i=1,\dots,k$.
Let $\hat{\boldsymbol{\rho}}^{C_{r+i}}$ be normalized version of
$\bar{\boldsymbol{\rho}}^{C_{r+i}}$ such that 
\begin{equation}
\mathbbm{1}^{T}\cdot\hat{\boldsymbol{\rho}}^{C_{r+i}}=1\,.\label{eq:normalize}
\end{equation}
Then the $i^{th}$ column of $R$ is defined as $R_{i}=\hat{\boldsymbol{\rho}}^{C_{r+i}}$
such that
\[
F\cdot R_{i}=\mathbf{0}\,.
\]
\nomenclature{$\mathcal{V}(X)$}{The set of vertice of digraph $X$}
As a result we have matrix $R$ which satisfies 
\begin{equation}
F\cdot R=\mathbf{0}\,.\label{eq:FR}
\end{equation}
In other words, for any vertex $j\not\in\mathcal{V}\left(C_{r+i}\right)$,
$\left(R_{i}\right)_{j}=0$. This in turn implies that $Y\equiv\mathbf{0}$
and $V$ is non-zero matrix corresponding to the tSCCs of complementary
digraph $G^{\star}$. 

Then we can construct the matrix $L$ using the following process.
Let $U$ be the matrix given by first transposing $V$ and then replacing
each nonzero element of $V$ by $1$. Since $\mathbbm{1}^{T}\cdot\mathcal{L}(C_{r+i})=\mathbbm{1}^{T}\cdot\mathcal{L}_{r+i}=\mathbf{0}$,
we have that $U\cdot T=\mathbf{0}$ and by  (\ref{eq:normalize})
we have that $U\cdot V=\mathbb{I}_{k}$. Let the matrix $X$ be constructed
as $X=-U\cdot B\cdot N^{-1}$. With this definition and (\ref{LR partition}),
we can see that the matrix $L$ satisfies the criteria $L\cdot F=\mathbf{0}$,
\begin{equation}
L\cdot F=\left(\begin{array}{c|c}
X & U\end{array}\right)\cdot\left(\begin{array}{c|c}
N & \mathbf{0}\\
\hline B & T
\end{array}\right)=\left(\begin{array}{c|c}
X\cdot N+U\cdot B & U\cdot T\end{array}\right)=\left(\begin{array}{c|c}
-U\cdot B\cdot N^{-1}\cdot N+U\cdot B & U\cdot T\end{array}\right)=\mathbf{0}\label{eq:LF}
\end{equation}
Moreover, by definition, (\ref{LR partition}), the matrices $L$
and $R$ also satisfy (\ref{eq:LR}) as illustrated by 
\begin{equation}
L\cdot R=\left(\begin{array}{c|c}
X & U\end{array}\right)\cdot\left(\begin{array}{c}
Y\\
\hline V
\end{array}\right)=X\cdot Y+U\cdot V=0+\mathbb{I}_{k}=\mathbb{I}_{k}\label{eq:LRI}
\end{equation}
Therefore by (\ref{eq:FR}), (\ref{eq:LF}) and (\ref{eq:LRI}) we
can conclude that constructed matrices $R$ and $L$ satisfies (\ref{eq:LR})
and so the ES of the SD dynamics (\ref{eq:3}) is given by (\ref{eq:SS solution}).

\subsection{\label{sec:Illustrarion-of-Results}Illustration of Results}

Consider the directed graph $G$ with $5$ vertices, given in Figure
\ref{fig: 3.3 Digraph G} with its complementary digraph $G^{\star}$
in Figure \ref{fig: 3.3 Complementary-digraph}.

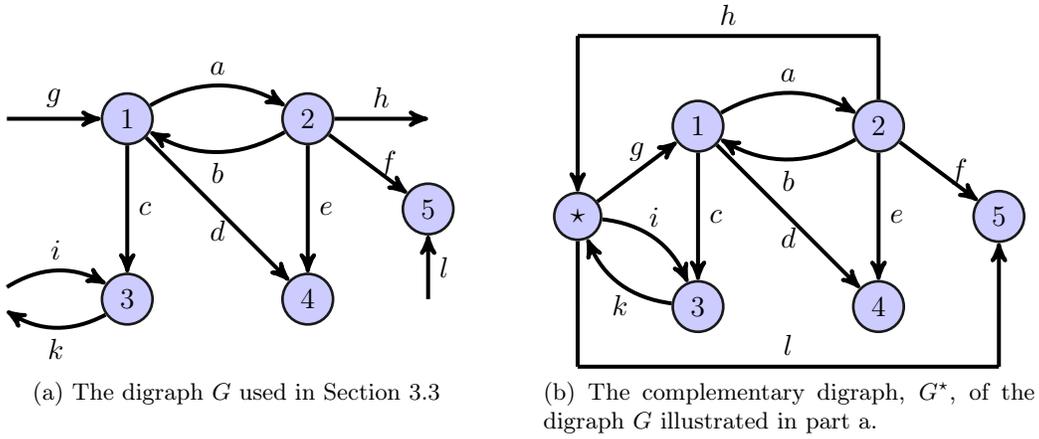
\begin{figure}
\centering\subfloat[The digraph $G$ used in Section \ref{sec:Illustrarion-of-Results}\label{fig: 3.3 Digraph G}]{\begin{tikzpicture}[scale=0.8, auto=left, every node/.style={circle,draw=black!90,fill=blue!20}, line width=1pt]  

\node (1) at (0,0) {1};   
\node (2) at (3,0)  {2};   
\node (3) at (0,-3)  {3};   
\node (4) at (3,-3)  {4};   
\node (5) at (5,-1.5)  {5};

\begin{scope}[thick, ->,>=stealth', every node/.style = {above}, line width=1.5pt]  

\draw  (1) edge [bend left] node [above]{$a$} (2);
\draw  (2) edge [bend left] node [below]{$b$} (1);
\draw  (1) edge  node [below]{$d$} (4);
\draw  (1) edge node [right]{$c$} (3); 
\draw  (2) edge  node [right]{$e$} (4);
\draw  (2) edge  node [right]{$f$} (5);
\draw  (2) edge  node [above]{$h$} (5,0);   
\path  (-2,0) edge  node [above] {$g$} (1);
\path  (-2,-2.8) edge [bend left] node [above]{$i$} (3);
\draw  (3) edge[bend left]  node [below]{$k$} (-2,-3.2); 
\path  (5,-3) edge  node [right]{$l$} (5); 
\end{scope}

\end{tikzpicture} 

}\hspace{1cm}\subfloat[The complementary digraph, $G^{\star}$, of the digraph $G$ illustrated
in part a. \label{fig: 3.3 Complementary-digraph}]{
\begin{tikzpicture}[scale=0.8, auto=left, every node/.style={circle,draw=black!90,fill=blue!20}, line width=1pt]  

\node (1) at (0,0) {1};   
\node (2) at (3,0)  {2};   
\node (3) at (0,-3)  {3};   
\node (4) at (3,-3)  {4};   
\node (5) at (5,-1.5)  {5}; 
\node (*) at (-2,-1.5)  {$\star$};  

\begin{scope}[thick, ->,>=stealth', every node/.style = {above}, line width=1.5pt]  

\draw  (1) edge [bend left] node [above]{$a$} (2);
\draw  (2) edge [bend left] node [below]{$b$} (1);
\draw  (1) edge  node [below]{$d$} (4);
\draw  (1) edge node [right]{$c$} (3); 
\draw  (2) edge  node [right]{$e$} (4);
\draw  (2) edge  node [right]{$f$} (5);
\draw[-]  (2) --(3,1.5);    
\draw[-]  (3,1.5)--(-2,1.5) node [midway, above] {$h$};  
\draw  (-2,1.5) -- (*);  
\draw  (*) edge  node [above]{$g$} (1);
\draw  (*) edge [bend left] node [above]{$i$} (3);
\draw  (3) edge[bend left]  node [below]{$k$} (*); 
\draw [-]  (*) --(-2,-4); 
\draw [-] (-2,-4)--(5,-4) node[midway, above] {$l$}; 
\draw  (5,-4)-- (5); 
\end{scope}

\end{tikzpicture}

}

\caption{A digraph with its corresponding complementary digraph. \label{fig:Illustration-of-results}}
\end{figure}

We have labeled vertices according to the labeling procedure described
in Section \ref{sec:Preliminary-results}. We note that the complementary
digraph $G^{\star}$ of digraph $G$ is not strongly connected, with
non-terminal SCC $C_{1}$ with $\mathcal{V}\left(C_{1}\right)=\{1,2,3,\star\}$
and two tSCCs, $C_{2}$ with $\mathcal{V}\left(C_{2}\right)=\left\{ 4\right\} $
and $C_{3}$ with $\mathcal{V}\left(C_{3}\right)=\left\{ 5\right\} $.
The Laplacian matrix for this digraph is 
\begin{equation}
\mathcal{L}(G)=\left(\begin{array}{ccccc}
-(a+c+d) & b & 0 & 0 & 0\\
a & -(b+e+f) & 0 & 0 & 0\\
c & 0 & 0 & 0 & 0\\
d & e & 0 & 0 & 0\\
0 & f & 0 & 0 & 0
\end{array}\right)\,.\label{eq:Laplacian of illustration}
\end{equation}

The degradation and synthesis are given by $D=diag\left(0,h,i,0,0\right)$,
$\mathbf{s}=\left(g,0,k,0,l\right)^{T},$ respectively. The associated
SD dynamics is given by a system of ODEs as in  (\ref{eq:3}). So
the matrix $F=\mathcal{L}(G)-D$ with its corresponding partitioning
is given by 
\begin{equation}
F=\left(\begin{array}{c|c}
N & 0\\
\hline B & T
\end{array}\right)=\left(\begin{array}{ccc|cc}
-(a+c+d) & b & 0 & 0 & 0\\
a & -(b+e+f+h) & 0 & 0 & 0\\
c & 0 & i & 0 & 0\\
\hline d & e & 0 & 0 & 0\\
0 & f & 0 & 0 & 0
\end{array}\right)\label{eq:illus F}
\end{equation}
Then we partition the vector $\mathbf{s}$ such that it matches with
the partitioning of $F$, so that we have 
\[
\mathbf{s}=\left(\begin{array}{c}
\mathbf{s}'\\
\hline \mathbf{s}''
\end{array}\right),\text{ where }\mathbf{s}'=\left(\begin{array}{c}
g\\
0\\
k
\end{array}\right)\text{ and }\mathbf{s}''=\left(\begin{array}{c}
0\\
l
\end{array}\right)\,.
\]
In order for an ES to exist, the vector $\mathbf{s}$ should satisfy
the necessary and sufficient conditions given in Theorem \ref{thm:BNS}.
By the first condition we have that $\mathbf{s}''\equiv\mathbf{0}$,
and thus $\boxed{l=0}$ implying that we cannot have synthesis on
vertex $5$. Next, we check the second condition of Theorem \ref{thm:BNS},
\[
\mathbf{0}=BN^{-1}\mathbf{s}'=\mathbbm{1}^{T}BN^{-1}\mathbf{s}'=\left(\begin{array}{ccc}
-\frac{a(e+f)+d(b+e+f+h)}{b(c+d)+(a+c+d)(e+f+h)} & -\frac{bd+(a+c+d)(e+f)}{b(c+d)+(a+c+d)(e+f+h)} & 0\end{array}\right)\left(\begin{array}{c}
g\\
0\\
k
\end{array}\right)\,,
\]
and conclude that $\boxed{g=0}$. Hence, the steady state solution
exists if and only if $g=l=0$. The digraphs $G$ and $G^{\star}$
which are compatible with having a steady state are depicted in Figure
\ref{fig:ES-compatible-digraph}. In a comparison with Figure \ref{fig: 3.3 Digraph G},
we note the absence of synthesis on nodes $1$ and $5$.
\begin{figure}
\centering\subfloat[ES compatible digraph $G$. Edges $g$ and $l$ were deleted.\label{fig: 3.3 Digraph G-1}]{\begin{tikzpicture}[scale=0.8, auto=left, every node/.style={circle,draw=black!90,fill=blue!20}, line width=1pt]  

\node (1) at (0,0) {1};   
\node (2) at (3,0)  {2};   
\node (3) at (0,-3)  {3};   
\node (4) at (3,-3)  {4};   
\node (5) at (5,-1.5)  {5};

\begin{scope}[thick, ->,>=stealth', every node/.style = {above}, line width=1.5pt]  

\draw  (1) edge [bend left] node [above]{$a$} (2);
\draw  (2) edge [bend left] node [below]{$b$} (1);
\draw  (1) edge  node [below]{$d$} (4);
\draw  (1) edge node [right]{$c$} (3); 
\draw  (2) edge  node [right]{$e$} (4);
\draw  (2) edge  node [right]{$f$} (5);
\draw  (2) edge  node [above]{$h$} (5,0);   
\path  (-2,-2.8) edge [bend left] node [above]{$i$} (3);
\draw  (3) edge[bend left]  node [below]{$k$} (-2,-3.2);
\end{scope}

\end{tikzpicture} 

}\hspace{1cm}\subfloat[Complementary digraph $G^{\star}$ \label{fig: 3.3 Complementary-digraph-1}]{\begin{tikzpicture}[scale=0.8, auto=left, every node/.style={circle,draw=black!90,fill=blue!20}, line width=1pt]  
\node (1) at (0,0) {1};   
\node (2) at (3,0)  {2};   
\node (3) at (0,-3)  {3};   
\node (4) at (3,-3)  {4};   
\node (5) at (5,-1.5)  {5}; 
\node (*) at (-2,-1.5)  {$\star$}; 
\begin{scope}[thick, ->,>=stealth', every node/.style = {above}, line width=1.5pt]  

\draw  (1) edge [bend left] node [above]{$a$} (2);
\draw  (2) edge [bend left] node [below]{$b$} (1);
\draw  (1) edge  node [below]{$d$} (4);
\draw  (1) edge node [right]{$c$} (3); 
\draw  (2) edge  node [right]{$e$} (4);
\draw  (2) edge  node [right]{$f$} (5);
\draw[-]  (2) --(3,1.5);   
\draw[-]  (3,1.5)--(-2,1.5) node [midway, above] {$h$}; 
\draw  (-2,1.5) -- (*);   
\draw  (*) edge [bend left] node [above]{$i$} (3);
\draw  (3) edge[bend left]  node [below]{$k$} (*); 
\end{scope}

\end{tikzpicture}

}

\caption{Steady state compatible digraph with its corresponding complementary
digraph. \label{fig:ES-compatible-digraph} The digraph $G$ illustrated
in Figure \ref{fig: 3.3 Digraph G} was tested for the conditions
of Theorem \ref{thm:BNS} and conflicting edges were deleted. }
\end{figure}
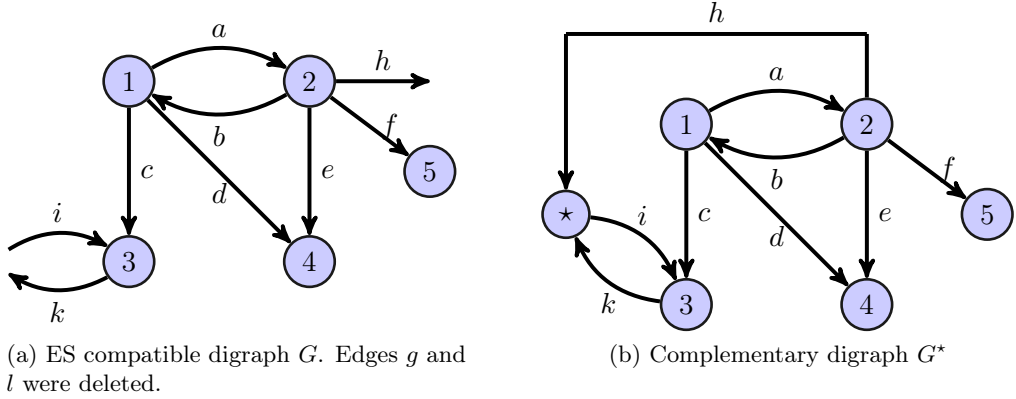

Every matrix is defined as in (\ref{eq:illus F}) except for the vector
$\mathbf{s}=(0,0,k,0,0)^{T}.$ As before $G^{\star}$ has two tSCCs,
$C_{2}$ with $\mathcal{V}\left(C_{2}\right)=\{4\}$ and $C_{3}$
with $\mathcal{V}\left(C_{3}\right)=\{5\}$. Now with the necessary
and sufficient conditions in hand, we will follow the construction
process for the matrices $R$ and $L$ described in Section \ref{sub:Construction-of-matrices}.
Since tSCCs, $C_{2}$ and $C_{3}$, each have only one vertex, we
have normalized vectors $\hat{\boldsymbol{\rho}}^{C_{2}}=(0,0,0,1,0)^{T}$
and $\hat{\boldsymbol{\rho}}^{C_{3}}=(0,0,0,0,1)$. So the columns
of the matrix $R$ are defined as $R_{1}=\hat{\boldsymbol{\rho}}^{C_{2}},\, R_{2}=\hat{\boldsymbol{\rho}}^{C_{3}}$,
\[
R=\left(\begin{array}{cc}
0 & 0\\
0 & 0\\
0 & 0\\
\hline 1 & 0\\
0 & 1
\end{array}\right)=\left(\begin{array}{c}
Y\\
\hline V
\end{array}\right)
\]
Therefore, $V=\left(\begin{array}{cc}
1 & 0\\
0 & 1
\end{array}\right)$, then transposing $V$ and writing $1$'s instead of its non-zero
elements we get $U=\left(\begin{array}{cc}
0 & 1\\
1 & 0
\end{array}\right)$. And so the matrix $X$ is 
\[
X=-U\cdot B\cdot N^{-1}=\frac{1}{b(c+d)+(a+c+d)(e+f+h)}\left(\begin{array}{ccc}
af & (a+c+d)f & 0\\
ae+d(b+e+f+h) & bd+(a+c+d)e & 0
\end{array}\right)
\]
So
\[
L=\left(\begin{array}{c|c}
X & U\end{array}\right)=\left(\begin{array}{ccc|cc}
\frac{af}{b(c+d)+(a+c+d)(e+f+h)} & \frac{(a+c+d)f}{b(c+d)+(a+c+d)(e+f+h)} & 0 & 0 & 1\\
\frac{ae+d(b+e+f+h)}{b(c+d)+(a+c+d)(e+f+h)} & \frac{bd+(a+c+d)e}{b(c+d)+(a+c+d)(e+f+h)} & 0 & 1 & 0
\end{array}\right)
\]
Accordingly, by  (\ref{eq:SS solution}) the ES, $\mathbf{x}_{{\scriptscriptstyle ES}}$,
is given by 
\[
\mathbf{x}_{{\scriptscriptstyle ES}}=R\cdot L\cdot\mathbf{x}_{0}+\left(R\cdot L-\mathbb{I}_{n}\right)\cdot F^{+}\cdot\mathbf{s}=R\cdot L\cdot\mathbf{x}_{0}+\left(\begin{array}{ccccc}
0 & 0 & \frac{k}{i} & 0 & 0\end{array}\right)^{T}\,.
\]

\section{\label{sec: negativity of inverse} Inverse of Non-Singular Perturbed\emph{
}Matrices}

In our previous paper \cite{MirzaevGunawardena2013bmb}, we have proven
that the perturbed Laplacian matrix of a strongly connected digraph
is non-singular. Nevertheless, we claimed that the inverse of such
a matrix has non-positive entries. In this section, we prove our claim
and provide a graph theoretic algorithm for the computation of the
inverse of perturbed matrices. Once again consider a digraph $G$
with $n$ nodes. As before Laplacian matrix for this digraph is given
by matrix $\mathcal{L}(G)\in\mathbb{R}^{n\times n}$ and perturbed
matrix is defined as $P=\mathcal{L}(G)-\Delta$, where $\Delta$ is
diagonal matrix with non-negative entries, 
\[
\left(\Delta\right)_{ij}=\begin{cases}
\delta_{i} & i=j\\
0 & i\not=j
\end{cases}\,.
\]
Remember from Remark \ref{lem:perturbed} that a perturbed matrix
of a strongly connected digraph is a non-singular matrix. However,
a perturbed matrix of an arbitrary digraph is not necessarily non-singular.
Here, we will prove that the inverse of any non-singular perturbed
matrix is a non-positive matrix. By that we mean all the elements
of the inverse matrix, $P^{-1}$, are non-positive real numbers (henceforth,
$P$ represents a non-singular perturbed matrix). To accomplish this
we first prove the statement for the case when the digraph $G$ is
strongly connected and then prove it for an arbitrary digraph.

\subsection{\label{sub:Strongly-connected-case inverse}Strongly connected case}

When the digraph $G$ is strongly connected we will use the explicit
formulation of the inverse of a non-singular matrix $P$, derived
from Laplace expansion of the determinant,

\begin{equation}
P^{-1}=\frac{1}{\det(P)}adj(P)\label{eq:adjugate}
\end{equation}
where $\left(adj(P)\right)_{ij}=(-1)^{i+j}P_{(ji)}$, the $ij$-th
entry of the adjugate being the $(ji)$-th minor of $P$ (up to the
sign). At this point we can reconstruct the $i$-th row and $j$-th
column of the matrix $P$ such that the constructed matrix is the
Laplacian matrix of a strongly connected digraph denoted as $G^{ij}$.
Then it follows that $\mathcal{L}(G^{ij})_{(ij)}=P_{(ij)}$. An example
illustration of such a construction is given in Figure \ref{fig:Reconstruction-of-the}.
Thus the $ij$-th minor of the matrix $P$ can be calculated as the
$ij$-th minor of the matrix $\mathcal{L}(G^{ij})$, which in turn
can be calculated using the MTT,

\begin{align}
\left(adj(P)\right)_{ji} & =(-1)^{i+j}P_{(ij)}=(-1)^{i+j}\mathcal{L}(G^{ij})_{(ij)}\nonumber \\
 & =(-1)^{i+j}(-1)^{n+i+j-1}\left(\boldsymbol{\rho}^{G^{ij}}\right)_{i}=(-1)^{n-1}\left(\boldsymbol{\rho}^{G^{ij}}\right)_{i}\,.\label{eq:constructed G}
\end{align}

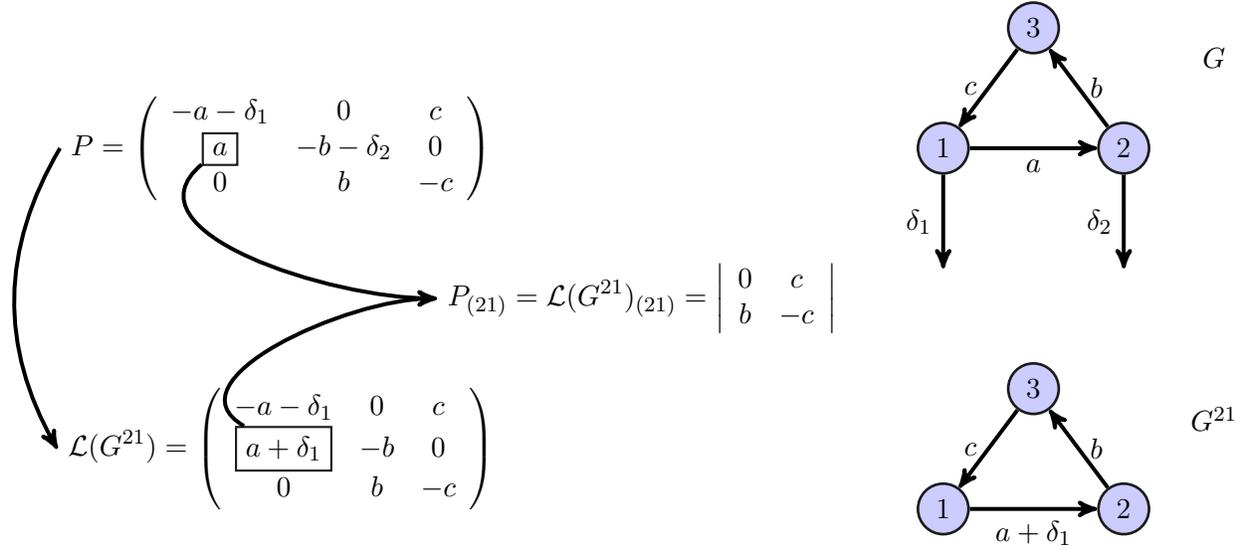
\begin{figure}
  {\tikzstyle{every picture}+=[remember picture]
 \begin{tikzpicture}[scale=0.8, auto=left, line width=1pt]  
  
\node (n1) at (-10,0) {$\displaystyle P=\left(\begin{array}{ccc} -a-\delta_{1} & 0 & c\\ \tikz[baseline]{\node[rectangle,draw=black!90,thick, anchor=base] (t1){$a$};} & -b-\delta_{2} & 0\\ 0 & b & -c \end{array}\right)$};

\node (n3) at (-10,-5) {$\displaystyle \mathcal{L}(G^{21})=\left(\begin{array}{ccc} -a-\delta_{1} & 0 & c\\ \tikz[baseline]{\node[rectangle,draw=black!90,thick, anchor=base] (t2){$a+\delta_1$};} & -b & 0\\ 0 & b & -c \end{array}\right)$};

\node (n4) at (-4,-2.5) {$\displaystyle P_{(21)}=\mathcal{L}(G^{21})_{(21)}=\left|\begin{array}{cc} 0 & c\\  b & -c \end{array}\right|$};

\node[ circle,draw=black!90,fill=blue!20] (1) at (1,0) {1};    
\node[ circle,draw=black!90,fill=blue!20] (2) at (4,0)  {2};    
\node[ circle,draw=black!90,fill=blue!20] (3) at (2.5,2)  {3};  
\node (55) at (5.5,1.5) {$G$};
\node[ circle,draw=black!90,fill=blue!20] (11) at (1,-6) {1};    
\node[ circle,draw=black!90,fill=blue!20] (22) at (4,-6)  {2};    
\node[ circle,draw=black!90,fill=blue!20] (33) at (2.5,-4)  {3};  
\node (44) at (5.5,-4.5) {$G^{21}$};   

\begin{scope}[thick, ->,>=stealth', every node/.style = {above}, line width=1.5pt]   
\draw  (1) edge  node [below]{$a$} (2); 
\draw  (2) edge  node [right]{$b$} (3); 
\draw  (3) edge  node [left][left]{$c$} (1); 
\draw  (1) edge  node [left]{$\delta_1$} (1,-2); 
\draw  (2) edge  node [left]{$\delta_2$} (4,-2);
\draw  (11) edge  node [below]{$a+\delta_1$} (22); 
\draw  (22) edge  node [right]{$b$} (33); 
\draw  (33) edge  node [left][left]{$c$} (11); 
\path[->] (t1) edge [out=220, in=180] (n4); 
\path[->] (t2) edge [out=150, in=180] (n4); 
\path[->]  (n1.west) edge [bend right] (n3.west);
\end{scope}

\end{tikzpicture}
}

\caption{$2$nd row and $1$st row of $P$ is reconstructed to get the Laplacian
matrix $\mathcal{L}(G^{21})$ and the associated strongly connected
digraph $G^{21}$. \emph{At the top (from left to right), }perturbed
matrix $P$ and the associated digraph $G$. \emph{In the middle,
}$\left(21\right)$-th minor of $P$ and $(21)$-th minor of $\mathcal{L}(G^{21})$.
\emph{At the bottom (from left to right), }Laplacian matrix $\mathcal{L}(G^{21})$
and the associated digraph $G^{21}$.\label{fig:Reconstruction-of-the} }
\end{figure}

Since the digraph $G^{ij}$ is strongly connected, from Proposition
\ref{prop:SCC kernel} we find that $(\boldsymbol{\rho}^{G^{ij}})_{i}>0$
for each index $i$. As for the determinant of the matrix $P$, we
can again add one row and one column to the matrix $P$ such that
the constructed $(n+1)\times(n+1)$ matrix is the Laplacian matrix
of strongly connected digraph $G^{\, n+1\, n+1}$. For the convenience
of the notation, hereafter, we refer to the digraph $G^{\, n+1\, n+1}$
simply as $G^{\, n+1}$. 

Then by\emph{ }the MTT it is implied that

\begin{equation}
\det(P)=\mathcal{L}(G^{\, n+1})_{(n+1\, n+1)}=(-1)^{3(n+1)-1}\left(\boldsymbol{\rho}^{G^{\, n+1}}\right)_{n+1}=(-1)^{n}\left(\boldsymbol{\rho}^{G^{\, n+1}}\right)_{n+1}\,.\label{eq:determinant}
\end{equation}

Note that the construction of such a row and a column is independent
of the existing rows, so we can always reconstruct a strongly connected
digraph $G^{\, n+1}$. Consequently, from Proposition \ref{prop:SCC kernel}
it follows that $(\boldsymbol{\rho}^{G^{n+1}})_{n+1}>0$. Therefore,\emph{
}(\ref{eq:adjugate})\emph{, }(\ref{eq:constructed G}) and (\ref{eq:determinant})
together imply that 

\[
(P^{-1})_{ij}=\frac{1}{det(P)}\left(adj(P)\right)_{ij}=\frac{(-1)^{n-1}\left(\boldsymbol{\rho}^{G^{ij}}\right)_{i}}{(-1)^{n}\left(\boldsymbol{\rho}^{G^{n+1}}\right)_{n+1}}=-\frac{\left(\boldsymbol{\rho}^{G^{ij}}\right)_{i}}{\left(\boldsymbol{\rho}^{G^{n+1}}\right)_{n+1}}<0\,,
\]
and thus all the entries of the matrix $P^{-1}$ are strictly less
than zero.

\subsection{General Case}

In the general case, the perturbed matrix of an arbitrary digraph
$G$ may not be a non-singular matrix. In Section \ref{sec:Theoretical-Development}
we have seen that for a perturbed matrix to be non-singular, the diagonal
blocks associated to each SCC should be a perturbed matrix on their
own. Consider an arbitrary digraph $G$ with $q$ SCCs, $C_{1},\cdots,C_{q}$,
we assume that the matrix $P$ can be partitioned analogous to (\ref{eq:partition}), 

\[
P=\left(\begin{array}{ccc}
\boxed{\mathcal{P}_{1}} & \cdots & \mathbf{0}\\
\vdots & \ddots & \vdots\\
+ & + & \boxed{\mathcal{P}_{q}}
\end{array}\right)\,,
\]
where each $\mathcal{P}_{i}=\mathcal{L}(C_{i})-\Delta_{i}$ is an
$a_{i}\times a_{i}$ perturbed matrix of SCC $C_{i}$. Having this
decomposition in hand, we are ready to prove that the inverse of the
matrix $P$ has all non-positive entries. To do that we will use the
results of Section \ref{sub:Strongly-connected-case inverse} and
follow the standard path for finding the inverse using the method
of Gaussian elimination,\begin{align*}
(\begin{array}{c:c}P & \mathbb{I}_{n}\end{array})=
&\left(\begin{array}{ccc:ccc}\boxed{\mathcal{P}_{1}} & \cdots & \mathbf{0} & \boxed{\mathbb{I}_{a_{1}}} & \cdots & \mathbf{0}\\\vdots & \ddots & \vdots & \vdots & \ddots & \vdots\\+ & + & \boxed{\mathcal{P}_{q}} & \mathbf{0} & \mathbf{0} & \boxed{\mathbb{I}_{a_{q}}}\end{array}\right)\\
\longrightarrow&\left(\begin{array}{ccc:ccc}\boxed{\mathbb{I}_{a_{1}}} & \cdots & \mathbf{0} & \boxed{\mathcal{P}_{1}^{-1}} & \cdots & \mathbf{0}\\\vdots & \ddots & \vdots & \vdots & \ddots & \vdots\\- & - & \boxed{\mathbb{I}_{a_{q}}} & \mathbf{0} & \mathbf{0} & \boxed{\mathcal{P}_{q}^{-1}}\end{array}\right)
\end{align*}

where minus sign, $-$, stands for some matrix with non-positive real
entries. From Section \ref{sub:Strongly-connected-case inverse} each
$\mathcal{P}_{i}^{-1}$ has negative entries. Hence multiplying a
block of rows having non-negative real numbers by $\mathcal{P}_{i}^{-1}$
transforms elements of those rows into non-positive real numbers. 

Consider the $ij$-the entry, $-a_{ij}$, of the LHS. Note that $jj$-th
entry of LHS is equal to $1$, e.g. $a_{jj}=1$. Then in order to
eliminate this negative element on the LHS of the augmented matrix
(i.e., $-a_{ij}$ ) we have to multiply the $j$-th row by a positive
real number, $a_{ij}$, and add resulting row to the $i$-th row.
Since all the elements of RHS are non-positive, this operation places
non-positive real numbers on $i$-th row of RHS. Performing operations
consecutively, on columns $1,2,\dots,n$ lead to the following matrix 

\begin{equation}
\longrightarrow\left(\begin{array}{ccc:ccc}\boxed{\mathbb{I}_{a_{1}}} & \cdots & \mathbf{0} & \boxed{\mathcal{P}_{1}^{-1}} & \cdots & \mathbf{0}\\\vdots & \ddots & \vdots & \vdots & \ddots & \vdots\\\mathbf{0} & \mathbf{0} & \boxed{\mathbb{I}_{a_{q}}} & - & - & \boxed{\mathcal{P}_{q}^{-1}}\end{array}\right)=\left(\begin{array}{c:c}\mathbb{I}_{n} & P^{-1}\end{array}\right)
\label{eq:general inverse}
\end{equation}As it can be observed from the above equation, (\ref{eq:general inverse}),
all the entries of the matrix $P^{-1}$ are non-positive real numbers,
as desired. 

Note that the inverse matrix $P^{-1}$ preserves lower-block diagonal
structure of the perturbed matrix $P$. Moreover, recall that the
matrix $N$ defined in (\ref{eq: Definition of N}) is also a non-singular
perturbed matrix. Hence, all the entries of the inverse matrix $N^{-1}$
are non-positive real numbers.

\subsection{\label{sec:Computation-of inverse of N}Symbolic computation of $P^{-1}$
based on the Matrix-Tree Theorem}

For a given set of constant edge weights numerical computation of
the inverse matrix $P^{-1}$ is challenging, and even more challenging
is symbolic computation of the inverse. Therefore, here we provide
a graph theoretic algorithm for the symbolic computation of $P^{-1}$.
The algorithm is again based on the MTT, and utilizes the theory developed
in Section \ref{sub:SC framework}. 

We start by introducing strictly positive synthesis edges at each
vertex, i.e. $\mathbf{s}=(s_{1},\dots,s_{n})^{T}\in\mathbb{R}_{>0}^{n\times1}$.
This makes the complementary digraph $G^{\star}$ strongly connected,
since the vertex $\star$ can be reached from any other vertex and
any vertex can be reached from $\star$.\textbf{ }After applying the
framework of the strongly connected case (Section \ref{sub:SC framework},
(\ref{eq: SC steady state MTT})) for the symbolic synthesis vector,
$\mathbf{s}\in\mathbb{R}_{>0}^{n\times1}$, we get the graph theoretic
representation of the ES, $\mathbf{p}$, i.e. 
\begin{equation}
\left(\mathbf{p}\right)_{i}=\frac{\left(\boldsymbol{\rho}^{G^{\star}}\right)_{i}}{\left(\boldsymbol{\rho}^{G^{\star}}\right)_{\star}}\,.\label{eq:ES for P}
\end{equation}
Consider the standard basis of $\mathbb{R}^{n}$
\[
\left\{ \mathbf{e}^{(i)}=(0,\dots,0,1,0,\dots,0)\right\} _{i=1}^{n}\,.
\]
Then, multiplying $P^{-1}$ by the vector $\mathbf{e}^{(i)}$ yields
the $i$-th column of the inverse matrix $P^{-1}$. Thus we construct
$P^{-1}$ by constructing its one column at a time. For that consider
the following combinations of specific synthesis edges,

\begin{equation}
\mathbf{s}^{(0)}=(1,\dots,1)^{T},\ ,\mathbf{s}^{(i)}=(1,\dots,1,2,1,\dots,1)^{T}\quad i=1,\dots,n\label{eq:Specific synthesis}
\end{equation}
where only $2$ is the $i$-th entry of the vector $\mathbf{s}^{(i)}$.
One can then easily observe that 
\[
\left\{ \mathbf{s}^{(i)}-\mathbf{s}^{(0)}=\mathbf{e}^{(i)}=(0,\dots,0,1,0,\dots,0)^{T}\ i=1,\dots,n\right\} 
\]
is standard basis for \textbf{$\mathbb{R}^{n}$}. On the other hand
substituting the vectors $\left\{ \mathbf{s}^{(0)},\mathbf{s}^{(1)},\dots,\mathbf{s}^{(n)}\right\} $
into (\ref{eq:ES for P}) gives rise to the steady states $\left\{ \mathbf{p}^{(0)},\mathbf{p}^{(1)},\dots,\mathbf{p}^{(n)}\right\} $,
respectively.%
\footnote{Note that we cannot substitute $\mathbf{e}^{(i)}$ directly to (\ref{eq:ES for P}),
because this would make $G^{\star}$ not strongly connected.%
} Then again these steady states can be computed algebraically using
\emph{ }(\ref{eq:4}),
\[
-P^{-1}\cdot\mathbf{s}^{(i)}=\mathbf{p}^{(i)}\qquad\forall i=0,1,\dots,n.
\]
This in turn can be simplified to 
\[
P^{-1}\cdot\left(\mathbf{s}^{(i)}-\mathbf{s}^{(0)}\right)=P^{-1}\cdot\mathbf{e}^{(i)}=\mathbf{p}^{(0)}-\mathbf{p}^{(i)}\,.
\]
Since $\left\{ \mathbf{e}^{(1)},\dots,\mathbf{e}^{(n)}\right\} $
is a standard basis for $\mathbb{R}^{n\times1}$, $i$-th column of
the matrix $P^{-1}$ is given by the vector $\mathbf{p}^{(0)}-\mathbf{p}^{(i)}$
or simply as 
\begin{equation}
P^{-1}=\left(\begin{array}{c|c|c|c}
\mathbf{p}^{(0)} & \mathbf{p}^{(1)} & \mathbf{\cdots} & \mathbf{p}^{(n)}\end{array}\right)\cdot\left(\begin{array}{cccc}
1 & 1 & \cdots & 1\\
-1 & 0 & \cdots & 0\\
0 & -1 & \ddots & \vdots\\
\vdots & 0 & \ddots & 0\\
0 & \vdots & \cdots & -1
\end{array}\right)\label{eq:Symbolic inverse}
\end{equation}
Note that since spanning trees rooted at vertex $\star$ cannot contain
any outgoing edge from vertex $\star$, the synthesis edges $\left\{ s_{1},\dots,s_{n}\right\} $
will not contribute to $\left(\boldsymbol{\rho}^{G^{\star}}\right)_{\star}$.
Hence, $\left(\boldsymbol{\rho}^{G^{\star}}\right)_{\star}$ remains
same for each substitution of $\mathbf{s}^{(i)}$. Consequently, we
don't have to calculate $\left(\boldsymbol{\rho}^{G^{\star}}\right)_{\star}$
each time and divide the other terms by it, we can just factor out
and perform the division at the end. 

We will illustrate the algorithm presented in this subsection with
a simple example. Consider the perturbed matrix 
\[
P=\left(\begin{array}{cc}
-a & 0\\
a & -b
\end{array}\right)
\]
which is simple a $2\times2$ matrix whose inverse is 
\[
P^{-1}=\frac{1}{ab}\left(\begin{array}{cc}
-b & 0\\
-a & -a
\end{array}\right)\,.
\]
On the other hand, by (\ref{eq:ES for P}) the symbolic ES is given
by 
\begin{equation}
\mathbf{p}=\frac{1}{ab}\left(\begin{array}{c}
s_{1}b\\
s_{2}a+s_{1}a
\end{array}\right)\,.\label{eq:illus P inverse}
\end{equation}
Then, substituting the synthesis vectors defined in (\ref{eq:Specific synthesis})
into (\ref{eq:illus P inverse}) we find that 
\[
\mathbf{p}^{(0)}=\frac{1}{ab}\left(\begin{array}{c}
b\\
2a
\end{array}\right),\ \mathbf{p}^{(1)}=\frac{1}{ab}\left(\begin{array}{c}
2b\\
3a
\end{array}\right),\ \mathbf{p}^{(2)}=\frac{1}{ab}\left(\begin{array}{c}
b\\
3a
\end{array}\right)\,.
\]
Thus the inverse matrix $P^{-1}$ is given by (\ref{eq:Symbolic inverse}),
\[
P^{-1}=\left(\begin{array}{c|c|c}
\mathbf{p}^{(0)} & \mathbf{p}^{(1)} & \mathbf{p}^{(3)}\end{array}\right)\cdot\left(\begin{array}{cc}
1 & 1\\
-1 & 0\\
0 & -1
\end{array}\right)=\frac{1}{ab}\left(\begin{array}{ccc}
b & 2b & b\\
2a & 3a & 3a
\end{array}\right)\cdot\left(\begin{array}{cc}
1 & 1\\
-1 & 0\\
0 & -1
\end{array}\right)=\frac{1}{ab}\left(\begin{array}{cc}
-b & 0\\
-a & -a
\end{array}\right)\,.
\]

\section{\label{sec:Biochemical-Network-Application}Biochemical Network Application}

In this section we will describe how the above developed framework
is useful for symbolic computation of the steady state solutions of
biochemical reaction networks.

\subsection{Secretion of insulin granules in $\beta$-cells}

One of the most prevalent diseases, diabetes mellitus (or simply diabetes)
is characterized by high level of blood glucose. Diabetes results
from either pancreas does not release enough insulin, or cells do
not respond to insulin produced with increased consumption of sugar,
or combination of both \cite{Barg2003}. Insulin is blood glucose-lowering
hormone produced, processed and stored in secretory granules by pancreatic
$\beta$-cells in Langerhans islets \cite{Olofsson2002}. Consequently,
secretory granules are released to extracellular space, which is regulated
by Ca\textsuperscript{2+}- dependent exocytosis \cite{Wollheim1981}.
Since diabetes is related to secretional malfunctions \cite{Rorsman2003},
studying mechanism of both normal and pathological insulin release
in molecular level is crucial for understanding of disease process.

\begin{figure}[t]
\centering\includegraphics[scale=0.6]{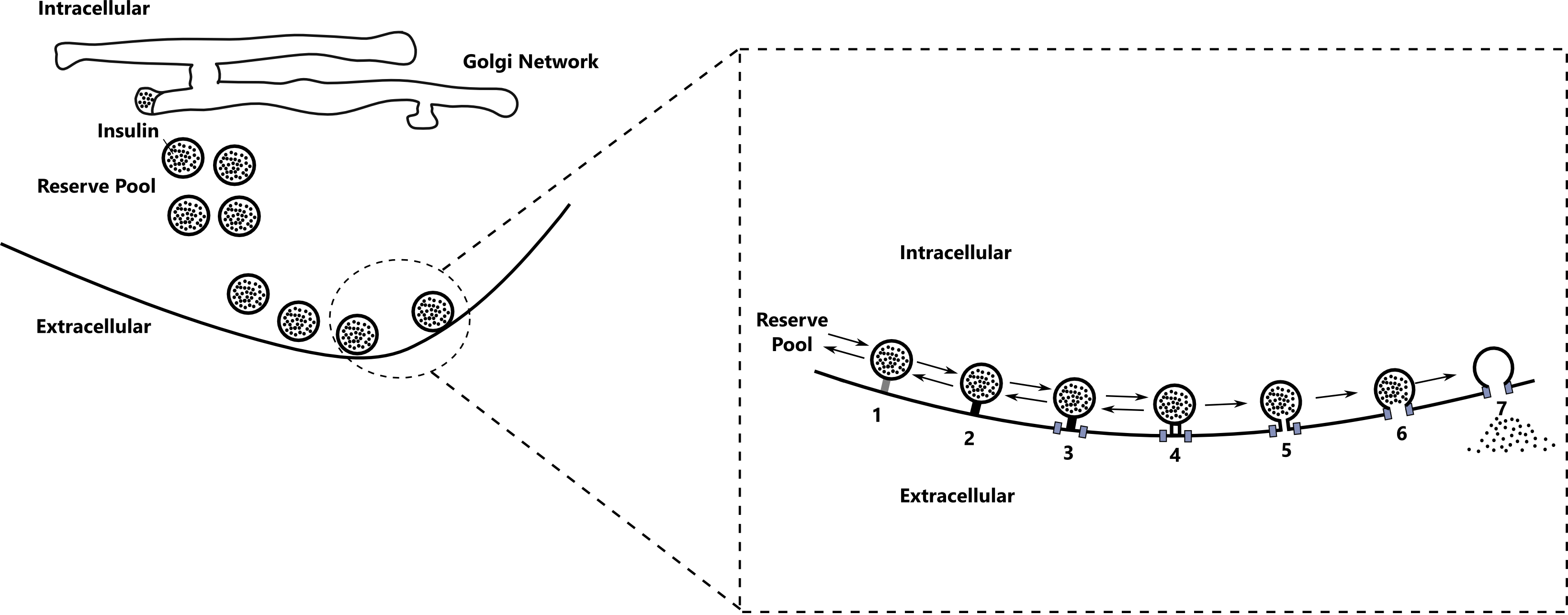}

\caption{Schematic drawing of exocytosis cascade in \textbf{$\beta$}-cells.
\emph{On the left, }insulin granules produced in Golgi network is
transported into extracellular space through exocytosis. \emph{On
the right, }particular steps involved in the exocytosis of the insulin
granules. The numbers stand for: 1) Re-supply 2) Priming 3) Domain
Binding 4) Ca Triggering 5) Fusion 6) Pore Expansion 7) Insulin Release
\label{fig:scheme}}
\end{figure}

Chen et al. \cite{Chen2008} developed mathematical model of $\beta$-cell
to calculate both rate of granule fusion and the rate of insulin secretion
in $\beta$-cells stimulated with electrical potential. The model
is based on five-state kinetic model of granule fusion proposed by
Voets et al. \cite{Voets1999}. Figure \ref{fig:scheme} illustrates
kinetic scheme proposed for exocytosis in pancreatic $\beta$-cells.
As it is shown in the figure the model accounts for steps involved
in exocytosis cascade such as re-suply, priming, domain binding, Ca\textsuperscript{2+}
triggering, fusion, pore expansion and insulin release. It is assumed
that L-type (not the R-type) voltage-sensitive Ca\textsuperscript{2+}-channels
are used for secretion of primed granules through cell membrane. During
this process ``microdomains'' with high Ca\textsuperscript{2+}
concentration are formed at the inner mouth of L-type channels (illustrated
as circles in Figure \ref{fig:scheme}). Concentration of Ca\textsuperscript{2+}
in cytosol and microdomain at time $t$ are denoted by $C_{i}(t)$
and $C_{md}(t)$, respectively. Since the number of granules are far
less than number of Ca\textsuperscript{2+}, it is also assumed that
dynamics of Ca\textsuperscript{2+} is independent of exocytosis cascade.
For further details we refer the reader to the original paper \cite{Chen2008}. 

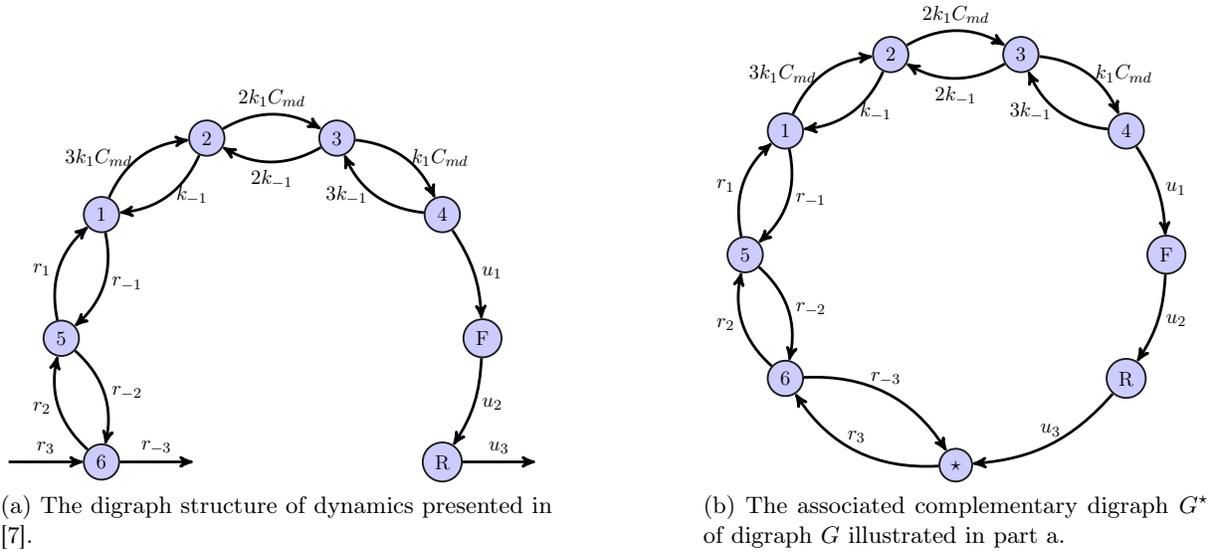
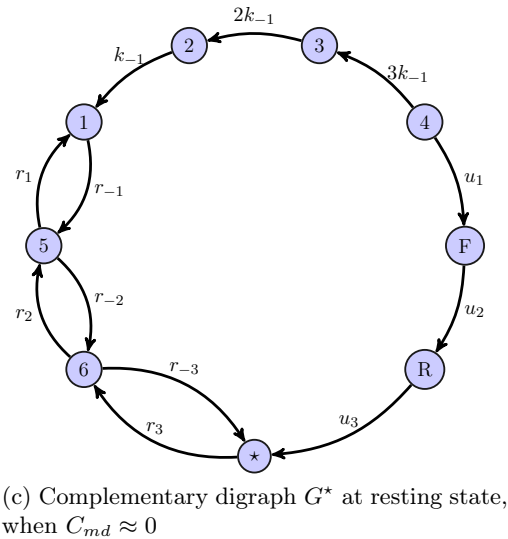
\begin{figure}[t]
\centering \subfloat[The digraph structure of dynamics presented in \cite{Chen2008}.\label{fig:Digraph Exocytosis}]{\scalebox{.7}{
\begin{tikzpicture}[scale=1,auto=left,every node/.style={circle,draw=black!90,fill=blue!20},   line width=1pt]  

\node (F) at (0:4) {F}; 
\node (4) at (1*360/10:4) {4}; 
\node (3) at (2*360/10:4) {3};
\node (2) at (3*360/10:4) {2};  
\node (1) at (4*360/10:4) {1};  
\node (5) at (5*360/10:4) {5}; 
\node (6) at (6*360/10:4) {6}; 
\node (R) at (9*360/10:4) {R};
\begin{scope}[thick,->,>=stealth', every node/.style = {above}, line width=1.5pt]                             \draw  (6) edge [bend left] node [left]{$r_2$} (5);    
\path (-5,-2.35) edge  node {$r_{3}$} (6);    
\draw (6) edge  node  {$r_{-3}$} (-1.5,-2.35);    
\draw  (5) edge [bend left] node [right]  {$r_{-2}$}(6) ;         
\draw  (5) edge [bend left] node [left]{$r_{1}$} (1);    
\draw  (1) edge [bend left] node [right] {$r_{-1}$} (5);         
\draw  (1) edge [bend left] node [left]{$3k_{1}C_{md}$}(2);     
\draw  (2) edge [bend left] node [right]{$k_{-1}$}(1);               
\draw  (2) edge [bend left] node [above]{$2k_{1}C_{md}$}(3);     
\draw  (3) edge  [bend left]node [below]{$2k_{-1}$}(2);               
\draw  (3) edge [bend left] node [right]{$k_{1}C_{md}$}(4);     
\draw  (4) edge  [bend left] node [left]{$3k_{-1}$}(3);        
\draw  (4) edge  [bend left=15]node [right] {$u_1$}(F);    
\draw  (F) edge [bend left=15]node [right]{$u_2$}(R) ;    
\draw (R) edge node [above]{$u_3$} (5,-2.35);   

\end{scope}
\end{tikzpicture}
}

}\hspace*{2cm}\subfloat[The associated complementary digraph $G^{\star}$ of digraph $G$
illustrated in part a. \label{fig:Complementary-digraph Exocytosis}]{\scalebox{.7}{
\begin{tikzpicture}[scale=1,auto=left,every node/.style={circle,draw=black!90,fill=blue!20},   line width=1pt]  

\node (F) at (0:4) {F}; 
\node (4) at (1*360/10:4) {4}; 
\node (3) at (2*360/10:4) {3};
\node (2) at (3*360/10:4) {2};  
\node (1) at (4*360/10:4) {1};  
\node (5) at (5*360/10:4) {5}; 
\node (6) at (6*360/10:4) {6}; 
\node (*) at (270:4) {$\star$};
\node (R) at (9*360/10:4) {R};
\begin{scope}[thick,->,>=stealth', every node/.style = {above}, line width=1.5pt]                             \draw  (6) edge [bend left] node [left]{$r_2$} (5);    
\path (*) edge [bend left] node {$r_{3}$} (6);    
\draw (6) edge [bend  left] node  {$r_{-3}$} (*);    
\draw  (5) edge [bend left] node [right]  {$r_{-2}$}(6) ;         
\draw  (5) edge [bend left] node [left]{$r_{1}$} (1);    
\draw  (1) edge [bend left] node [right] {$r_{-1}$} (5);         
\draw  (1) edge [bend left] node [left]{$3k_{1}C_{md}$}(2);     
\draw  (2) edge [bend left] node [right]{$k_{-1}$}(1);               
\draw  (2) edge [bend left] node [above]{$2k_{1}C_{md}$}(3);     
\draw  (3) edge  [bend left]node [below]{$2k_{-1}$}(2);               
\draw  (3) edge [bend left] node [right]{$k_{1}C_{md}$}(4);     
\draw  (4) edge  [bend left] node [left]{$3k_{-1}$}(3);        
\draw  (4) edge [bend left=15] node [right] {$u_1$}(F);    
\draw  (F) edge [bend left=15]node [right]{$u_2$}(R) ;    
\draw (R) edge [bend left=22] node [above]{$u_3$} (*);   

\end{scope}
\end{tikzpicture}
}

}

\subfloat[Complementary digraph $G^{\star}$ at resting state, when $C_{md}\approx0$\label{fig:Resting state}]{\scalebox{.7}{
\begin{tikzpicture}[scale=1,auto=left,every node/.style={circle,draw=black!90,fill=blue!20},   line width=1pt]  

\node (F) at (0:4) {F}; 
\node (4) at (1*360/10:4) {4}; 
\node (3) at (2*360/10:4) {3};
\node (2) at (3*360/10:4) {2};  
\node (1) at (4*360/10:4) {1};  
\node (5) at (5*360/10:4) {5}; 
\node (6) at (6*360/10:4) {6}; 
\node (*) at (270:4) {$\star$};
\node (R) at (9*360/10:4) {R};
\begin{scope}[thick,->,>=stealth', every node/.style = {above}, line width=1.5pt]                             \draw  (6) edge [bend left] node [left]{$r_2$} (5);    
\path (*) edge [bend left] node {$r_{3}$} (6);    
\draw (6) edge [bend  left] node  {$r_{-3}$} (*);    
\draw  (5) edge [bend left] node [right]  {$r_{-2}$}(6) ;         
\draw  (5) edge [bend left] node [left]{$r_{1}$} (1);    
\draw  (1) edge [bend left] node [right] {$r_{-1}$} (5);       
\draw  (2) edge [bend right=15] node {$k_{-1}$}(1);     
\draw  (3) edge  [bend right=15]node {$2k_{-1}$}(2);        
\draw  (4) edge  [bend right=15] node [right]{$3k_{-1}$}(3);        
\draw  (4) edge [bend left=15] node [right] {$u_1$}(F);    
\draw  (F) edge [bend left=15]node [right]{$u_2$}(R) ;    
\draw (R) edge [bend left=22] node [above]{$u_3$} (*);   

\end{scope}
\end{tikzpicture}
}

}

\caption{Exocytosis cascade of insulin granules in pancreatic $\beta$-cells\label{fig:Exocytosis-cascade} }
\end{figure}

Figure \ref{fig:Digraph Exocytosis} illustrates the dynamics associated
with exocytosis cascade as a digraph $G$. Since complementary digraph
of $G$ (Figure \ref{fig:Complementary-digraph Exocytosis}) is strongly
connected, the steady state solutions of dynamics can be calculated
using the algorithm described in Section \ref{sub:SC framework}.
The ES is given as 

\[
N_{{\scriptscriptstyle ES}}=\Delta\left(\begin{array}{c}
6k_{-1}^{3}+2u_{1}k_{-1}^{2}+C_{\text{md}}k_{1}u_{1}k_{-1}+2C_{\text{md}}^{2}k_{1}^{2}u_{1}\\
18C_{\text{md}}k_{1}k_{-1}^{2}+6C_{\text{md}}k_{1}u_{1}k_{-1}+3C_{\text{md}}^{2}k_{1}^{2}u_{1}\\
18C_{\text{md}}^{2}k_{-1}k_{1}^{2}+6C_{\text{md}}^{2}u_{1}k_{1}^{2}\\
6C_{\text{md}}^{3}k_{1}^{3}\\
\frac{6r_{-1}k_{-1}^{3}}{r_{1}}+\frac{2r_{-1}u_{1}k_{-1}^{2}}{r_{1}}+\frac{C_{\text{md}}k_{1}r_{-1}u_{1}k_{-1}}{r_{1}}+\frac{6C_{\text{md}}^{3}k_{1}^{3}u_{1}}{r_{1}}+\frac{2C_{\text{md}}^{2}k_{1}^{2}r_{-1}u_{1}}{r_{1}}\\
\frac{6r_{-2}r_{-1}k_{-1}^{3}}{r_{1}r_{2}}+\frac{2r_{-2}r_{-1}u_{1}k_{-1}^{2}}{r_{1}r_{2}}+\frac{C_{\text{md}}k_{1}r_{-2}r_{-1}u_{1}k_{-1}}{r_{1}r_{2}}+\frac{6C_{\text{md}}^{3}k_{1}^{3}u_{1}}{r_{2}}+\frac{6C_{\text{md}}^{3}k_{1}^{3}r_{-2}u_{1}}{r_{1}r_{2}}+\frac{2C_{\text{md}}^{2}k_{1}^{2}r_{-2}r_{-1}u_{1}}{r_{1}r_{2}}\\
\frac{6C_{\text{md}}^{3}k_{1}^{3}u_{1}}{u_{2}}\\
\frac{6C_{\text{md}}^{3}k_{1}^{3}u_{1}}{u_{3}}
\end{array}\right)
\]

where $\Delta$ is given as follows
\[
\begin{array}{c}
\Delta=\frac{r_{1}r_{2}r_{3}}{r_{-1}r_{-2}r_{-3}\left(\frac{6k_{1}^{3}r_{1}u_{1}C_{\text{md}}^{3}}{r_{-2}r_{-1}}+\frac{6k_{1}^{3}r_{1}r_{2}u_{1}C_{\text{md}}^{3}}{r_{-3}r_{-2}r_{-1}}+\frac{6k_{1}^{3}u_{1}C_{\text{md}}^{3}}{r_{-1}}+k_{1}k_{-1}u_{1}C_{\text{md}}+2k_{1}^{2}u_{1}C_{\text{md}}^{2}+2k_{-1}^{2}u_{1}+6k_{-1}^{3}\right)}\end{array}
\]

As we can see the ES gets complicated for the large graphs. However,
our framework provides steady state value of any given substrate (see
(\ref{eq: SC steady state MTT})), which is not easily found by numerical
simulations. 

At the resting state (electric potential set to $V=-70\, mV$ ), concentration
of Ca\textsuperscript{2+} in the microdomain is very low, so it is
assumed that $C_{md}=C_{md}(t)\approx0$. In this case, the complementary
digraph of $G$ is no longer strongly connected, which is given in
Figure \ref{fig:Resting state}. Then the ES solution have to be computed
by the process described in Section \ref{sub:Construction-of-matrices},
and is given as 
\[
N_{r,{\scriptscriptstyle ES}}=\left(\begin{array}{c}
\frac{r_{1}r_{2}r_{3}}{r_{-3}r_{-2}r_{-1}}\\
0\\
0\\
0\\
\frac{r_{2}r_{3}}{r_{-3}r_{-2}}\\
\frac{r_{3}}{r_{-3}}\\
0\\
0
\end{array}\right)
\]

In the above example the dynamics were essentially linear. Although
the framework in this paper is linear in nature, it can be applied
to nonlinear systems as well. This can be done by incorporating nonlinearity
into the framework through the edge labels. So far we treated edge
weights as uninterpreted symbols. In fact, edge weights can be an
arbitrary positive rational expressions. For example, the Michaelis-Menten
formula used in enzyme kinetics is a legitimate edge weight 
\[
a=\frac{V_{max}[S]}{K_{m}+[S]}\,,
\]
where $[S]$ stands for the concentration of the substrate $S$, $K_{m}$
and $V_{max}$ are reaction specific constants. However, in most cases
a chemical reaction network modeled with mass action kinetics, which
gives rise to a nonlinear system of ODEs. The steady states of this
type of dynamics can also be algorithmically computed using our framework.
For instance, a chemical reaction of type 
\[
A+B\overset{k}{\longrightarrow}C
\]
can be represented in our way as 
\[
A\overset{kB}{\longrightarrow}C
\]
One can then use above formalism to transform chemical reactions into
a digraph with time dependent edge weights. Consequently, this digraph
can be used to calculate the steady states using our framework. One
should keep in mind that in a such transformation only the equilibrium
solutions coincide not the transient dynamics \cite{Gunawardena2012}.
For more extensive discussion of the topic we refer the reader to
\cite{Gunawardena2012} and \cite{Gunawardena2014}. Next we illustrate
such incorporation by applying it to a nonlinear biochemical network.

\subsection{Glucose metabolism in $\beta$-cells}

In their paper Sweet and Matschinsky \cite{Sweet1995} setup a mathematical
model of pancreatic $\beta$-cell glucose metabolism to investigate
the relation between glucose and the rate of glycolysis (see Figure
\ref{fig:Schematic-glycolysis}). Since glucose metabolism in $\beta$-cells
indirectly affects the rate of insulin secretion \cite{Pedersen2009},
this type of models have implications for the diabetes treatment.
All reactions together make dynamics overwhelmingly complex. To avoid
this authors assumed that reactions inside dashed rectangles (pools)
are operating sufficiently fast, and have reached thermodynamic equilibrium.
Then ordinary differential equations are written for the rates of
transfers between these pools. Consequently, equilibrium metabolites
in a pool are calculated algebraically using equilibrium assumptions.
Although the model is minimalistic, it still includes parameters describing
overall behavior glycolysis in $\beta$-cell. Nonlinear dynamics associated
with the model can be given as in figure \ref{fig:Glycolysis-in-pancreatic}A.
In this case nonlinearity is hidden in label $f$, 
\[
f=\frac{G_{PK}D\sqrt{b^{2}-b-\frac{8K_{TPI}}{V_{c}}{\scriptstyle GIP(t)}}}{{\scriptstyle GIP(t)}}
\]
where ${\scriptstyle GIP(t)}$ stands for concentration of ${\scriptstyle GIP}$
at time $t$, and every other letter are reaction rate constants.
Since complementary digraph of $G$ (Figure \ref{fig:Glycolysis-in-pancreatic}b)
is strongly connected, the ES solutions to the system can be obtained
by the procedure described in Section \ref{sub:SC framework}:
\[
\left(\begin{array}{c}
Glu^{*}\\
H-6-P^{*}\\
\frac{1}{2}GIP^{*}\\
Pyr^{*}
\end{array}\right)=\frac{1}{bdh+beh+ceh}\left(\begin{array}{c}
adh+aeh\\
ach\\
\frac{aceh}{f}\\
ace+bdg+beg+ceg
\end{array}\right)
\]

\begin{figure}[t]
\centering\includegraphics[scale=0.4]{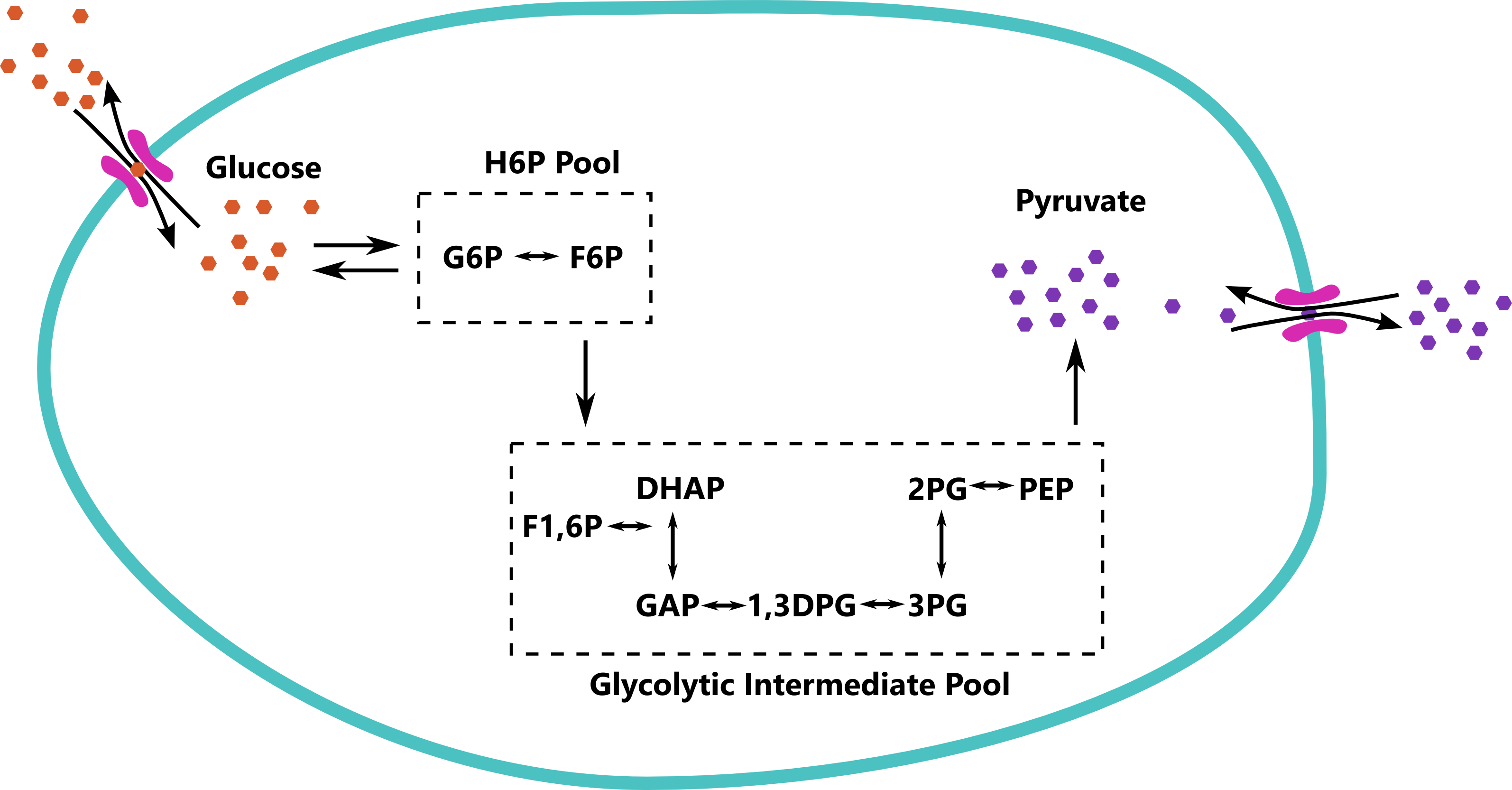}

\caption{Schematic model of glycolysis in pancreatic $\beta$-cells \label{fig:Schematic-glycolysis}}

\end{figure}

\begin{figure}
\centering\subfloat[The digraph describing the main mechanism of glycolisis]{\scalebox{.7}{
\begin{tikzpicture}[scale=1,auto=left,every node/.style={circle,draw=black!90,fill=blue!20},   line width=1pt] 

\tikzstyle{ann} = [draw=none,fill=none,right]

\node (P) at (30:4) {Pyr}; 
\node (GP) at (310:4) {$\frac{1}{2}$GIP}; 
\node (H) at (230:4) {H-6-P};
\node (G) at (150:4) {Glu};  
\node[ann]  (in) [above of=G, node distance=2.5cm]{};
\node[ann]  (out) [above of=P, node distance=2.5cm]{};

\begin{scope}[thick,->,>=stealth', every node/.style = {above}, line width=1.5pt]      
 \path  (in) edge [bend left] node[right] {$a$} (G);    
\path (G) edge [bend left] node [left]{$b$} (in);
\path  (G) edge [bend left] node [right] {$c$} (H);    
\path (H) edge [bend left] node [left]{$d$} (G);
\path (H) edge [bend right=15] node[above]  {$e$} (GP);
\path (GP) edge [bend right=20] node[left] {$f$} (P);  
\path (out) edge [bend left] node [right]{$g$} (P);
\path (P) edge [bend left] node [left]{$h$} (out);
\end{scope}
\end{tikzpicture}
}

}\hspace*{2cm}\subfloat[The corresponding complementary digraph $G^{\star}$]{\scalebox{.7}{
\begin{tikzpicture}[scale=1,auto=left,every node/.style={circle,draw=black!90,fill=blue!20},   line width=1pt]  
\tikzstyle{ann} = [draw=none,fill=none,right]

\node (P) at (30:4) {Pyr}; 
\node (GP) at (310:4) {$\frac{1}{2}$GIP}; 
\node (H) at (230:4) {H-6-P};
\node (G) at (150:4) {Glu};  
\node (*) at (90:4) {$\star$};

\begin{scope}[thick,->,>=stealth', every node/.style = {above}, line width=1.5pt]      
 \path  (*) edge [bend left] node[right] {$a$} (G);    
\path (G) edge [bend left] node [left]{$b$} (*);
\path  (G) edge [bend left] node [right] {$c$} (H);    
\path (H) edge [bend left] node [left]{$d$} (G);
\path (H) edge [bend right=20] node[above]  {$e$} (GP);
\path (GP) edge [bend right=25] node[left] {$f$} (P);  
\path (*) edge [bend left=35] node {$g$} (P);
\path (P) edge [bend left=35] node [below] {$h$} (*);
\end{scope}
\end{tikzpicture}
}

}

\caption{Glycolysis in pancreatic $\beta$-cells\label{fig:Glycolysis-in-pancreatic}}
\end{figure}
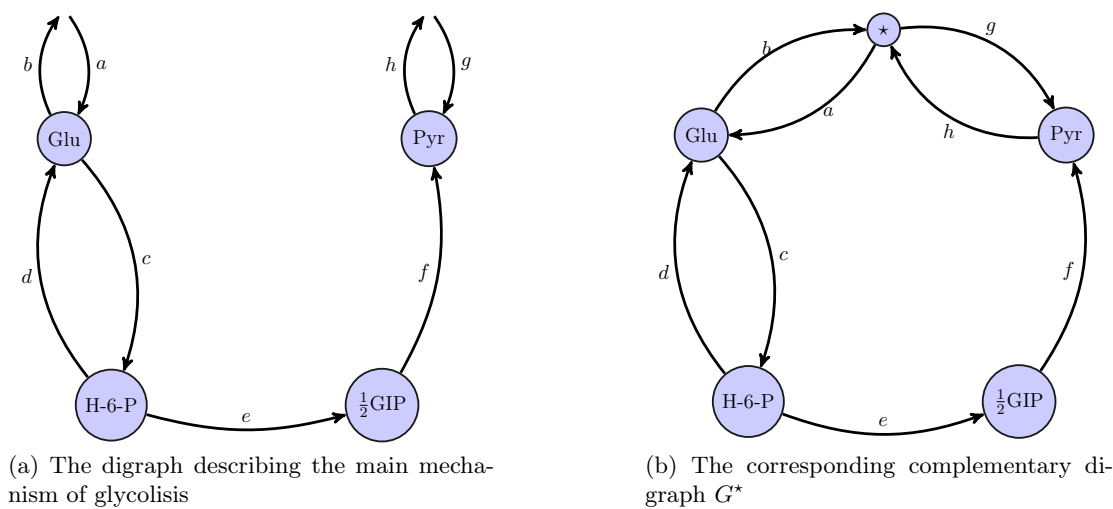

\section{\label{sec:Conclusions}Conclusions and future work}

In our previous work, we have developed a ``linear framework'' for
symbolical computation of equilibrium solutions of Laplacian dynamics,
which has applications in many diverse fields of biology such as enzyme
kinetics, pharmacology and receptor theory, gene regulation, protein
post-translation modification \cite{Gunawardena2012,Gunawardena2014,Ahsendorf2013}.
Our effort here was to extend existing framework for the case when
zeroth order synthesis and first order degradation is added to Laplacian
dynamics. The main motivation came from \cite{Gunawardena2012}, where
the author discusses the addition of synthesis and degradation to
Laplacian dynamics of strongly connected digraph. Here we extended
the proposed framework for arbitrary digraph with synthesis and degradation,
and showed that synthesis and degradation dynamics possesses unique
stable steady state solution under certain necessary and sufficient
conditions. These conditions can be also used to identify whether
given synthesis and degradation dynamics reaches a steady state. Moreover,
as before, we have developed a mathematical framework to compute that
unique ES. Our algorithm uses underlying digraph structure of dynamics
and computer implementation of previous framework \cite{Ahsendorf2013}
can be revised for automatic computations. 

This type of dynamics are frequently encountered in biological literature.
In fact, to illustrate utility of our framework we have applied it
to several examples in biochemistry such as exocytosis cascade of
insulin granules and glucose metabolism in pancreatic $\beta$-cells.
Since computed steady states are exact (not an approximation), they
can be used to check correctness of numerical solutions. On the other
hand, one of the greatest challenges in mathematical modeling is finding
required parameters using given set of experimental data. Yet another
feature of framework is that it can prove useful in parameter estimation
problems. Particularly, one can calibrate computed symbolic ES solutions
to experimental results.

Although the latter example, glucose metabolism in $\beta$-cells,
demonstrates application of framework to nonlinear system of differential
equations, the scope of application of our framework to such nonlinear
systems is limited. Therefore, as our future plan we intend to further
extend framework such that it can be applied to broader range of nonlinear
dynamics.

\section{Acknowledgements}

Funding for this research was supported in part by grants NIH-NIGMS
2R01GM069438-06A2 and NSF-DMS 1225878. The authors would also like
to thank Clay Thompson (Systems Biology Group, Pfizer, Inc.) for his
suggestion of the insulin synthesis example used in Section \ref{sec:Biochemical-Network-Application}.

\printnomenclature[5cm]{}

\bibliographystyle{siam}
\bibliography{mathbioCU}

\end{document}